\documentclass[
bibliography=totoc,parskip,fontsize=10pt]{scrartcl}

\usepackage[margin=1.1in]{geometry}  
\usepackage{graphicx}              
\usepackage{amsmath}               
\usepackage{amsfonts}              
\usepackage{amsthm}                
\usepackage{amssymb}
\usepackage{paralist}
\usepackage[usenames,dvipsnames]{color} 
\usepackage[british]{babel}
\usepackage{cite}
\usepackage[affil-it]{authblk}
\usepackage{xypic}
\usepackage{setspace}                
\usepackage{hyperref}

\usepackage[automark]{scrpage2}		
\clearscrheadfoot
\ohead{}
\lohead{\rightmark}
\rehead{\leftmark}

\ohead{\textbf{\pagemark}}

\setheadsepline{1pt} 

\addtokomafont{section}{\center}

\newtheorem{Thm}{Theorem}
\newtheorem{thm}{Theorem}[section]
\newtheorem{lem}[thm]{Lemma}
\newtheorem{prop}[thm]{Proposition}
\newtheorem{cor}[thm]{Corollary}

\newtheorem*{defi}{Definition}
\newtheorem*{rem}{Remark}

\newtheorem*{Th}{Theorem}

\newtheorem*{exa}{Example}

\DeclareMathAlphabet\mathbb{U}{fplmbb}{m}{n}

\newcommand{\quest}[1]{\textbf{Question #1.}}

\newcommand{\OO}{\mathbb{O}}
\newcommand{\HH}{\mathbb{H}}
\newcommand{\CC}{\mathbb{C}}
\newcommand{\RR}{\mathbb{R}}     
\newcommand{\QQ}{\mathbb{Q}}
\newcommand{\ZZ}{\mathbb{Z}}     
\newcommand{\NN}{\mathbb{N}}

\newcommand{\defeq}{\mathrel{\mathop{\raisebox{1.3pt}{\scriptsize$:$}}}=}
\newcommand{\eqdef}{\mathrel{=\!\!\mathop{\raisebox{1.3pt}{\scriptsize$:$}}}}

\newcommand\opna{\operatorname}
\newcommand\mf{\mathfrak}
\newcommand\mc{\mathcal}
\newcommand\mass{\operatorname{mass}}

\addtokomafont{section}{ \large \rmfamily\scshape}
\addtokomafont{subsection}{\normalsize \rmfamily\scshape}

\setkomafont{sectioning}{\normalcolor\rmfamily\scshape}

\begin{document}

\title{\Large Filling Invariants of Stratified Nilpotent Lie Groups}
\date{}
\author{\large Moritz Gruber}
\maketitle

\begin{abstract}\noindent\small
\textbf{Abstract. }Filling invariants are measurements of a metric space describing the behaviour of isoperimetric inequalities. In this article we examine filling functions and higher divergence functions. We prove for a class of stratified nilpotent Lie groups that in the low dimensions the 
filling functions grow as fast as the ones of the Euclidean space and in the high dimensions slower than the filling functions of the Euclidean space. We do this by developing a purely algebraic condition on the Lie algebra of a stratified nilpotent Lie group. Further, we find a sufficient criterion for such groups to have a filling function in a special dimension with faster growth as the appropriate filling function of the Euclidean space . Further we bound the higher divergence functions of stratified nilpotent Lie groups.
\end{abstract}

%
%

%
%

\pagenumbering{arabic}

\setcounter{secnumdepth}{2}


\section{Introduction}
\vspace*{-5mm}

Filling invariants of metric spaces measure the difficulty of finding admissible fillings of given boundaries.
These invariants are of interest in geometric group theory because their asymptotic growth rates describe the large scale geometry of a metric space.\\
the first of these invariants we are interested in are the \textit{filling functions}. They measure the difficulty to fill Lipschitz cycles by Lipschitz chains.
In the case of the Euclidean space the $(m+1)^{th}$ filling function grows like $\ell^\frac{m+1}{m}$. Similar, Wenger proved that the $(m+1)^{th}$ filling function of a Hadamard space does not grow faster than $\ell^\frac{m+1}{m}$ (see \cite{Wenger05}). And more explicit, Leuzinger showed in the case of symmetric spaces of non-compact type the $(m+1)^{th}$ filling function grows exactly as $\ell^\frac{m+1}{m}$ as long as $m$ is smaller than the rank of the symmetric space and it grows linearly in the dimensions above (see \cite{Leuzinger14}).\\
The second family of filling invariants we are interested in are the \textit{higher divergence functions}. They measure the difficulty to fill an outside an $r$-ball lying Lipschitz cycle with an outside a $\rho r$-ball, $0 < \rho \le 1$, lying Lipschitz chain. For them the situation is much the same as in the case of the filling functions: Their behaviour is well understood for non-positively curved spaces (see for example \cite{BradyFarb}, \cite{Leuzinger2000}, \cite{Wenger06}).\\
So, leaving the world of non-positively curved spaces suggests itself for finding new interesting results. As strictly positively curved spaces are of bounded diameter and so quasi-isometric to points, one has to look at spaces with the whole spectrum of curvature. A rich class of such spaces form the nilpotent Lie groups. These Lie groups have all three types of curvature  in every point (see Wolf \cite{Wolf}).
Burillo and Young computed the filling functions of the complex Heisenberg Groups (see \cite{Burillo}, \cite{Young1} and \cite{YoungII}). One of the in the proof mainly used properties of these groups is that their Lie algebras allow gradings. This property generalises to the class of \textit{stratified nilpotent Lie groups}.\\
We show that such stratified nilpotent Lie groups are suitable to apply the techniques of Burillo and Young, if their Lie algebras satisfy a purely algebraic condition.
So we prove a similar division of Euclidean, super- and sub-Euclidean growth of the filling functions of such stratified nilpotent Lie groups as in the case of the complex Heisenberg Groups. As application, we will see that the Heisenberg Groups over the Hamilton quaternions and over the octonions satisfy the conditions for our theorems.\\
The results presented in this paper were part of the the author's dissertation \cite{Doktorarbeit} at the Karlsruhe Institute of Technology.

\section{Preliminaries}
\vspace*{-5mm}

In this section we collect some of the concepts, facts and notation which will be used in this paper.

\vspace*{-5mm}
\subsubsection{Filling functions}\label{SectionFF}
\vspace*{-5mm}

Filling invariants measure the difficulty to fill a given boundary. The \textit{$(m+1)^{th}$ filling function} does this by describing the difficulty of filling Lipschitz $m$-cycles by Lipschitz $(m + 1)$-chains. 

Let $X$ be a metric space and $m \in \NN$. We denote by $\mc H^m$ the $m$-dimensional Hausdorff-measure of $X$.  The \textit{$m$-dimensional volume} of a subset $A\subset X$ is $\opna{vol}_m(A)\defeq \mc H^m(A)\ .$
Further we denote by $\Delta^m$ the $m$-simplex equipped with an Euclidean metric.
\begin{defi}
A \textup{Lipschitz $m$-chain} $a$ in $X$ is a (finite) formal sum $a=\sum_j z_j\alpha_j$ of Lipschitz maps $\alpha_j:\Delta^m \to X$ with coefficients $z_j\in \ZZ$. \\
The \textup{boundary} of a Lipschitz $m$-chain $a=\sum_j z_j\alpha_j$ is defined as the Lipschitz $(m-1)$-chain 
$$\partial a=\sum_j \big(z_j \sum_{i=0}^m (-1)^i \alpha_{j|\Delta^m_i} \big)$$
where $\Delta^m_i$ denotes the $i^{th}$ face of $\Delta^m$.\\
A Lipschitz $m$-chain $a$ with zero-boundary, i.e. $\partial a=0$, is called a \textup{Lipschitz $m$-cycle}. \\
A \textup{filling} of a Lipschitz $m$-cycle $a$ is a Lipschitz $(m+1)$-chain $b$ with boundary $\partial b=a$. \\
We define the $\textup{mass}$ of a Lipschitz $m$-chain $a$ as the total volume of its summands: 
$$\mass (a)\defeq \sum_j z_j \opna{vol}_m(\alpha_j(\Delta^m))\ .$$
\end{defi}
If $X$ is a Riemannian manifold, the volume of such a summand is given by $\opna{vol}_m(\alpha_j(\Delta^m))=\int_{\Delta^m} J_{\alpha_j} \opna d \lambda$, where $ \opna d \lambda$ denotes the $m$-dimensional Lebesgue-measure and $J_{\alpha_j}$ is the jacobian of $\alpha_j$. This is well defined, as Lipschitz maps are, by Rademacher's Theorem, almost everywhere differentiable.  

Given a $m$-cycle, one is interested in the best possible filling of it, i.e. the filling with the smallest mass. To deduce a property of the space $X$, one varies the cycle and examines how large the ratio of the mass of the optimal filling and the mass of the cycle can get. This leads to the \textit{filling functions}:

\begin{defi}
Let $n \in \mathbb N$ and let $X$ be a $n$-connected metric space.
For $m\le n$ the \textup{$(m+1)^{th}$ filling function} of $X$ is given by
$$F^{m+1}_X(l)=\sup_a \inf_b \mass(b) \qquad \forall l \in \RR^+,$$
where the infimum is taken over all $(m+1)$-chains $b$ with $\partial b=a$ and the supremum is taken over all  $m$-cycles $a$ with $\mass(a)\le l$.
\end{defi}

As our main interest lies in the large scale geometry of the space $X$, the exact description of the filling functions is of less importance to us. Indeed we only look at the asymptotic growth rate of this functions. We do this by the following equivalence relation, which makes the growth rate of the filling functions an quasi-isometry invariant.

\begin{defi}
Let $f,g:\mathbb R^+ \to \mathbb R^+$ be functions. Then we write $f\preccurlyeq g$ if there is a constant $C>0$ with 
$$f(l) \le Cg(Cl)+Cl+C \quad \forall l \in \mathbb R^+.$$
If $f\preccurlyeq g$ and $g\preccurlyeq f$ we write $f\sim g$. This defines an equivalence relation.
\end{defi}

We read this notation $f \preccurlyeq g$ as ``\textit{$f$ is bounded from above by $g$} " respectively ``\textit{$g$ is bounded from below by $f$} " according whether we are more interested in $f$ or $g$.

\begin{prop}[{see for example \cite[Lemma 1]{YoungH}}]
Let $X$ and $Y$ be $n$-connected metric spaces. Then holds:
$$X \text{ quasi-isometric to }Y \ \Rightarrow\  F^{j+1}_X \sim F^{j+1}_Y \quad \forall j \le n.$$ 
\end{prop}

Let's look at the example of the filling functions of the $n$-dimensional Euclidean space.

\begin{exa}[{see \cite{FF60}}]
The filling functions of the Euclidean space $\RR^n$ are
$$F^{j+1}_\RR(l) \sim l^\frac{j+1}{j} \quad \text{for } \ j \le n-1.$$ 
\end{exa}

This enables us to use the terms \textit{Euclidean, sub-Euclidean} respectively \textit{super-Euclidean filling function} for filling functions with the same, strictly slower respectively strictly faster growth rate
than $l^\frac{j+1}{j}$.

%
%

\subsubsection{Higher divergence functions}\label{SectionHDF}
\vspace*{-5mm}
Another family of filling invariants are the \textit{higher divergence functions}. They examine the asymptotic geometry of a space by studying quantitatively the topology at infinity. Roughly speaking, they measure the difficulty to fill an outside an $r$-ball lying $m$-cycle with an outside a $\rho r$-ball lying $(m+1)$-chain (for some $0<\rho \le1$).\\
\quad\\
Let $X$ be a simply connected metric space with basepoint $\mathfrak o \in X$ and $r>0$. 
%
We call a Lipschitz chain $a$ in $X$ \textit{$r$-avoidant}, if \ $\opna{image}(a)\cap B_r(\mathfrak o)=\emptyset$. 

\begin{figure}[h]
\centering
\includegraphics[width=80mm]{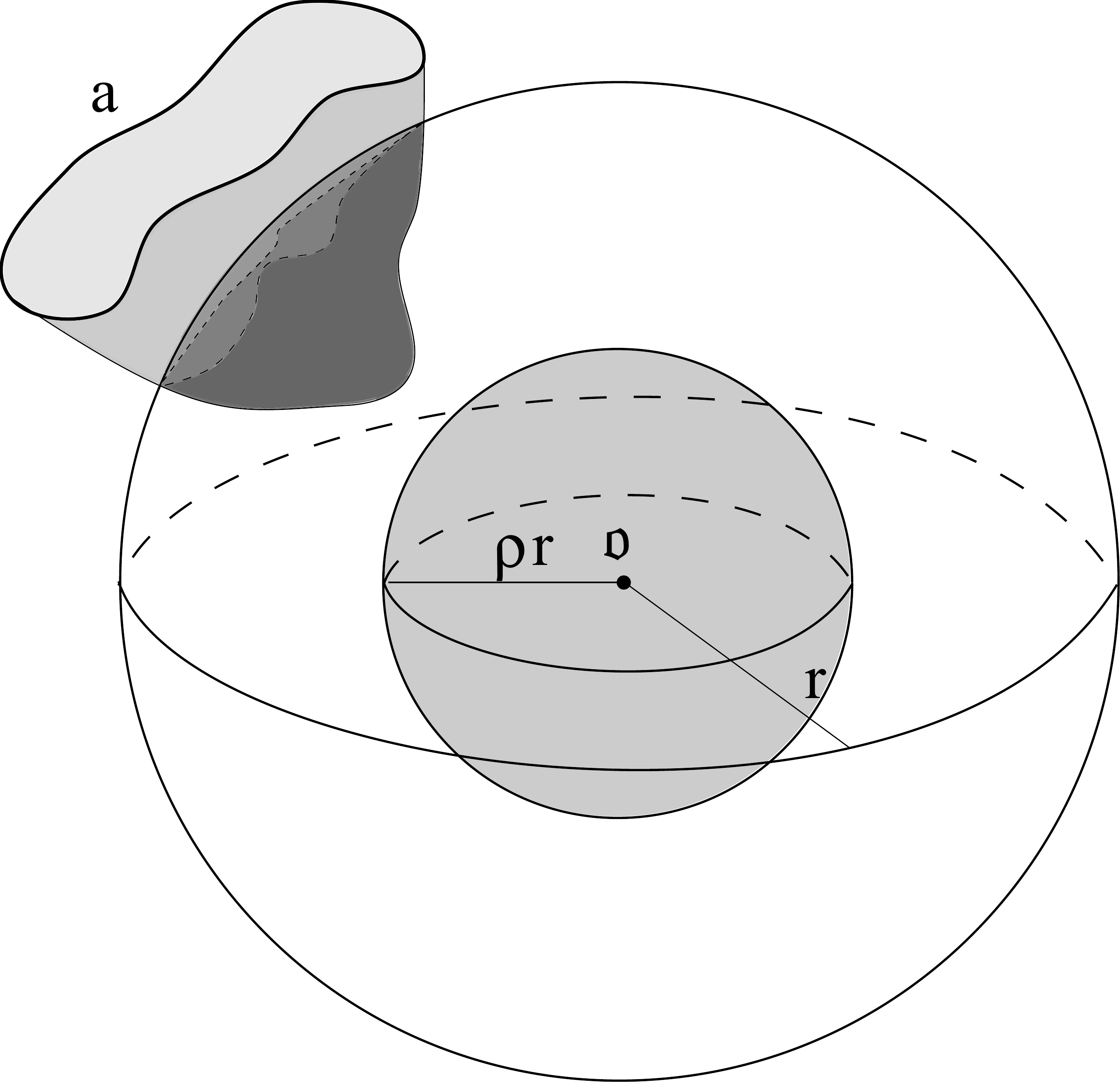}
\caption{An $r$-avoidant cycle \textbf{a} with a $\rho r$-avoidant filling (compare \cite{BradyFarb}).}
\end{figure}
\quad\\
One now wants to fill $r$-avoidant cycles by (nearly) $r$-avoidant chains. To do this, we need the cycle to be a boundary. In contrast to the case of the filling functions, here it doesn't suffice that $X$ is highly connected as the avoidant-condition can be imagined as deleting the $r$-ball around the basepoint. This leads to the following definition:

\begin{defi}
For $0<\rho\le 1$ we call $X$ $(\rho,n)$-\textup{acyclic at infinity}, if every $r$-avoidant Lipschitz $j$-cycle $a$ has a $\rho r$-avoidant filling for all $0\le j\le n$, i.e. there is a $\rho r$-avoidant Lipschitz $(j+1)$-chain $b$ with $\partial b=a$.\\
The \textup{divergence dimension} $\opna{divdim}(X)$ of $X$ is the largest integer $n$, such that $X$ is $(\rho,n)$-acyclic at infinity for some $\rho$.
\end{defi}

It can be easily seen, that the divergence dimension is always smaller than $\dim X -2$, as there are $(\dim X-2)$-cycles homotopic to the boundary of the $r$-ball $B_r(\mf o)$ around $\mf o$. These cycles, of course, are not boundaries of chains in $X \setminus B_r(\mf o)$.

In the following let $m$ be always less or equal to the divergence dimension of $X$. 

The following definition contains some technical parameters, which provide the tolerance needed to make the growth type of higher divergence functions a quasi-isometry invariant. 

\begin{defi}
For $0<\rho\le 1$ and $\alpha >0$ we set
$$\opna{div}^m_{\rho,\alpha}(r)\defeq \sup_a \opna{div}_\rho^m(a,\alpha r^m)\defeq \sup_a\inf_b \mass(b) \ \quad \forall r \in \RR^+,$$
where the infimum is taken over all $\rho r$-avoidant $(m+1)$-chains $b$ with $\partial b=a$ and the supremum is taken over all $r$-avoidant $m$-cycles $a$ with $\mass(a)\le \alpha r^m$.\\
\quad \\
Then the $m^{th}$-\textup{divergence function} of $X$ is the 2-parameter family
$$\opna{Div}^m(X)\defeq \{\opna{div}^m_{\rho,\alpha}\}_{\rho,\alpha}\ .$$
\end{defi}

Above we asked $m$ to be less or equal to the divergence dimension. Alternatively one could set the infimum over the empty set as $\infty$. In this case, the divergence dimension is the biggest number $n \in \NN$, such that there is an $\rho \in (0,1]$ with $\opna{div}^j_{\rho,\alpha} < \infty$ for all $j \le n$.

The functions $\opna{div}_{\rho,\alpha}^m$ are very explicit in terms of the metric. For example if one scales the metric by a constant $c>0$ the functions will scale to $\opna{div}_{\rho,\alpha}^m(c \ \cdot )$. As we are mostly interested in the asymptotic behaviour, we look at the equivalence classes of the higher divergence functions $\opna{Div}^m(X)$ under the below defined equivalence relation for special $2$-parameter families of functions. This makes the equivalence class of $\opna{Div}^m(X)$ an quasi-isometry invariant.

\begin{defi}\quad
\vspace{-5mm}
\begin{enumerate}[a)]
\item A \textup{positive 2-parameter $m$-family} is a 2-parameter family $F=\{f_{s,t}\}$ of functions $f_{s,t}:\RR^+\longrightarrow \RR^+$, indexed over $\ 0<s\le 1,t>0$,   together with a fixed integer $m$. \vspace{-3mm}
\item Let $m\in \mathbb N$ and let $F=\{f_{s,t}\}$ and $H=\{h_{s,t}\}$ be two positive 2-parameter $m$-families, indexed over $\ 0<s\le 1,t>0$.\\
Then we write $F\preceq H$, if there exists thresholds $0<s_0\le 1$ and $t_0>0$, as well as constants $L,M\ge 1$, such that for all $s\le s_0$ and all $t\ge t_0$ there is a constant $A\ge 1$ with:
\vspace{-2mm}
$$f_{s,t}(x)\le A h_{Ls,Mt}(Ax+A)+O(x^m) \ .$$
If both $F\preceq H$ and $H\preceq F$, so we write $F \sim H$. This defines an equivalence \\relation.\\
\end{enumerate}
\end{defi}\vspace{-13.5mm}
\quad\\
We read this notation $F \preccurlyeq H$ as ``\textit{$F$ is bounded from above by $H$}'' respectively \\``\textit{$H$ is bounded from below by $F$} '' depending on whether we are more interested in $F$ or $H$.

A special case is that of a positive 2-parameter $m$-family $F=\{f_{s,t}\}_{s,t}$ bounded from above (or below) by a constant positive 2-parameter $m$-family $H$, i.e. $H=\{h\}_{s,t}$. This means that all functions $f_{s,t}$ are bounded from above (or below) by the same growth type.
As this will often appear in the following, we just write $h$ for the constant positive 2-parameter $m$-family $\{h\}_{s,t}$.

From now on we consider $\opna{Div}^m(X)$ as positive 2-parameter $m$-family, indexed by $\rho$ and $\alpha$.

The relation "$\preceq$" (and consequently "$\sim$") only captures the asymptotic behaviour of the functions for $r \to \infty$\ :
Let $h:\RR_+ \to\RR_+$ be an increasing function. If  $B\ge1$ and $\opna{div}^m_{\rho,\alpha}(B)\le h(B)$ then $\opna{div}^m_{\rho,\alpha}(r)\le B\cdot h(B+B r)  $ for all $r \le B$, because both sides are increasing. So we need to examine the relation "$\preceq$" (and consequently "$\sim$")  only for $r$ larger than an arbitrary constant $B=B(\rho,\alpha)\ge 1$.

The proof of the following proposition can be found in \cite[Prop. 2.2]{ABDDY}. It uses the fact that one can vary the parameter $\rho$ by multiplying the constant $L$ in the equivalence relation and that the constant $A$ is chosen after the parameters $\rho$ and $\alpha$ (at this point there is an error in the appropriate definition of equivalence in \cite{ABDDY}).

\begin{prop} \label{QI}
Let $n\in \mathbb{N}$ and le $X,Y$ be metric spaces with basepoints and let $\opna{divdim}(X)\ge n$ and $\opna{divdim}(Y)\ge n$. Then holds: 
$$X \text{ quasi-isometric to }Y \ \Rightarrow\  \opna{Div}^j(X)\sim \opna{Div}^j(Y) \quad \forall j \le n.$$ 
\end{prop}

Again we look at the example of the $n$-dimensional Euclidean space and its higher divergence functions.
\begin{exa}[{see\cite{ABDDY}}]
The higher divergence functions of the Euclidean space $\RR^n$ are 
$$\opna{Div}^j_{\RR^n}(r) \sim r^{j+1} \quad \text{for }\  j \le n-2=\opna{divdim}(\RR^n).$$
\end{exa}

As in the case of the filling functions, this enables us to use the terms \textit{Euclidean, sub-Euclidean} respectively \textit{super-Euclidean $j^{th}$-divergence} for $j^{th}$-divergence functions with the same, strictly slower respectively strictly faster growth rate
than $r^{j+1}$.

%
%

\subsubsection{Filling invariants for integral currents} \label{SectionFIFIC}
\vspace*{-5mm}
We will use integral currents to prove our results for the filling invariants for Lipschitz chains. In order to do this, we modify already existing results concerning filling invariants for integral currents. So we have to introduce isoperimetric inequalities for integral currents. This is mainly done by replacing the words 'Lipschitz chain' by 'integral current' in the respective definitions of filling functions.\\
For an elaborated introduction to integral currents see \cite{AmbrosioKirchheim}.

We denote the set of $m$-dimensional integral currents on a complete metric space $X$ by $\boldsymbol I_m(X)$.

\begin{defi}
Let $X$ be a complete metric space and let $m \in \NN$. Then \textup{$X$ satisfies an isoperimetric inequality of rank $\delta$ for $\boldsymbol I_m(X)$}, if there is a constant $C>0$, such that for every integral current $T\in \boldsymbol I_m(X)$ with $\partial T =0$, there exists an integral current $S \in \boldsymbol I_{m+1}(X)$ with $\partial S=T$ and
$$\boldsymbol M(S) \le C \cdot \boldsymbol M(T)^\delta\ .$$
\end{defi}

%
%

\subsubsection{Stratified nilpotent Lie groups}\label{SectionSNLG}
\vspace*{-5mm}

A Lie group $G$ with Lie algebra $\mf g$ is called \textit{nilpotent}, if its lower central series 
$$G=G_1 \rhd G_2 \rhd G_3 \rhd ... \quad \text{ with }\ G_{j+1}=[G,G_j]$$
determines to the trivial group in finitely many steps. Here the bracket $[G, G_j]$ denotes the commutator group, i.e. the group generated by all commutators of elements of $G$ and $G_j$. This condition is equivalent to the condition that the lower central series of the Lie algebra 
$$\mf g = \mf g_1 \ge \mf g_2 \ge \mf g_3 \ge ... \quad \text{ with }\ \mf g_{j+1}=[\mf g, \mf g_j]$$
determines in finitely many steps to the null-space. Here the bracket $[\mf g, \mf g_j]$ denotes the linear subspace of $\mf g$ generated by all brackets of element of $\mf g$ and $\mf g_j$. In both cases the minimal number of steps in the lower central series needed to arrive at the trivial group or at the null-space, respectively, is the same, say $d$. It is called the \textit{degree of nilpotency} of $G$ and $\mf g$. For brevity we call a nilpotent Lie group of nilpotency degree $d$ in the following short \textit{$d$-step nilpotent} Lie group.\\
Our main concern is for a special class of nilpotent Lie groups, the \textit{stratified nilpotent Lie groups}. Their advantage is, that they additionally admit a family of self-similarities which are automorphisms. Further these self-similarities have nice properties concerning left-invariant (sub-)Riemannian metrics on the group. 
\begin{defi}
A \textup{stratified nilpotent Lie group} is a simply connected $d$-step nilpotent Lie group $G$  with Lie algebra $\mf g$ together with a grading
$$\mf g = V_1 \oplus V_2 \oplus ... \oplus V_d$$
with $[V_1,V_j]=V_{1+j}$ where $V_m=0$ if $m>d$.  
\end{defi}

For example, every simply connected $2$-step nilpotent Lie group $G$ is such a stratified nilpotent Lie group with grading $\mf g= V_1 \oplus [\mf g, \mf g]$, where $V_1$ is isomorphic to $\mf g / [\mf g,\mf g]$.

Recall that on a Lie group $G$ any two left-invariant Riemannian metrics are equivalent. This means, if $g$ and $g'$ are left-invariant Riemannian metrics on $G$, then there is a constants $L>0$, such that 
$$\frac{1}{L}\cdot g \le g' \le L\cdot g\ .$$
From this it follows directly, that $(G,g)$ and $(G,g')$ are quasi-isometric. So for our purpose to understand the asymptotic geometry of $G$, both metrics lead to the same results. Therefore it doesn't matter which left-invariant Riemannian metric we choose.\\
Most of time in which we will work explicitly with the Riemannian metric, we will choose, for technical reasons, a left-invariant Riemannian metric such that $V_i$ is orthogonal to $V_j$ whenever $i\ne j$. We call such a metric \textit{fitting to the grading}.

Now we can introduce the above mentioned self-similarities on a stratified nilpotent Lie group.

\begin{defi}
Let $G$ be a stratified nilpotent Lie group with grading $\mf g=V_1 \oplus ... \oplus V_d$ of its Lie algebra.
For every $t>0$ we define the map 
$$\hat s_t: \mf g \to \mf g \ , \ \hat s_t(v_j)\defeq  t^jv_j \quad \text{ for } \ v_j \in V_j.$$
As $\hat s_t$ is an automorphism of the Lie algebra, there is an uniquely determined automorphism $s_t: G \to G$ of the Lie group $G$ with $L(s_t):=d_es_t=\hat s_t$. We call this automorphism $s_t$  \textup{scaling automorphism}.
\end{defi}

The elements of the first layer $V_1$ are called \textit{horizontal}. As they are scaled least of all elements of the Lie algebra $\mf g$, they play an outstanding role.

\begin{defi}\quad \vspace{-5mm}
\begin{enumerate}[a)]
\item Let $M$ be a Riemannian manifold, $G$ be a stratified nilpotent Lie group and $f:M \longrightarrow G$ be a Lipschitz map. \\Then $f$ is \textup{horizontal}, if all the tangent vectors of its image lie in the subbundle 
$$\mathcal H\defeq \bigcup_{g\in G}dL_g V_1$$ 
of the tangent bundle of $G$.
\vspace{-3mm}
\item Let $X$ be a simplicial complex and $f:X \longrightarrow G$ be a Lipschitz map. Then $f$ is $m$-\textup{horizontal}, if $f$ is horizontal on the interior of all $j$-simplices, $j\le m$, of $X$.\\
\end{enumerate}
\end{defi}

\begin{lem}[{see \textup{\cite[2.2.1]{CDPT}}}]
Let $G$ be a stratified nilpotent Lie group.\vspace{-3mm}
\begin{enumerate}[a)]
\item To be horizontal is a left-invariant property, i.e. if $f$ is a horizontal map, then $L_g\circ f$ is a horizontal map for all $g\in G$.
\vspace{-3mm}
\item To be horizontal is invariant under scaling automorphisms, i.e. if $f$ is a horizontal map, then $s_t\circ f$ is a horizontal map for all $t>0$.

\end{enumerate}
\end{lem}

We now equip $G$ with a left-invariant Riemannian metric $g$ fitting to the grading of $\mf g$ with associated length metric $\opna d_g$. Then we get the following scaling estimates:
$$\opna d_g(s_t(p),s_t(q))\begin{cases} \le t \cdot \opna d_g(p,q) \quad \text{for } t\le1 \\  \ge t \cdot \opna d_g(p,q)\quad \text{for } t\ge1 \end{cases} \quad \forall p,q\in G.$$
In the above inequalities holds equality in both cases if and only if the distance of $p$ and $q$ is realised by a piecewise horizontal path. So we get in this special case:
$$\opna d_g(s_t(p),s_t(q))= t \cdot \opna d_g(p,q) \quad \forall t>0.$$
This leads to the following important property:

\begin{cor} \label{masslem}
Let $G$ be a stratified nilpotent Lie group with Riemannian metric $g$ fitting to the grading of $\mf g$ and let $a$ be a horizontal Lipschitz $m$-chain in $(G,\opna d_g)$. Further let $t>0$ and $s_t:G \to G$ be a scaling automorphism.
Then holds: 
$$\mass(s_t(a))=t^m \cdot \mass(a)\ .$$
\end{cor}

%
%
On a  stratified nilpotent Lie group there is another interesting metric. It is called the \textit{Carnot-Carath\'eodory metric}. It is the left-invariant sub-Riemannian metric $\opna d_c$ induced by $\mc H$. This means it is the length metric defined by the length with respect to the Riemannian metric $g$ of horizontal curves: 
$$d_c(p,q)\defeq \inf \{\opna{Length}(c)\mid c \text{ piecewise horizontal curve with } c(0)=p,\ c(1)=q \}$$ 
A stratified nilpotent Lie group equipped with its Carnot-Carath\'eodory metric is called a \textit{Carnot group}.
For the Carnot-Carath\'eodory metric, the nicest possible scaling behaviour holds:
$$\opna d_c(s_t(p),s_t(q))=t \cdot \opna d_c(p,q) \quad \forall p,q\in G.$$
The Carnot-Carath\'eodory metric is a length metric using the same length functional as the Riemannian distance. Further is the class of admissible curves a subset of all the piecewise smooth curves which are admissible for the Riemannian distance.
So we get the following relation between the two metric spaces $(G,\opna d_c)$ and $(G,\opna d_g)$, which later will become important: 

\begin{lem} \label{1lip}
Let $G$ be a stratified nilpotent Lie group with left-invariant Riemannian metric $g$ and associated length metric $\opna d_g$ and induced Carnot-Carath\'eodory metric $\opna d_c$.
Then the identity map 
$$\iota: (G,\opna d_c) \to (G, \opna d_g),\ x \mapsto x$$
is $1$-Lipschitz, i.e.
$$\opna d_g(x,y) \le \opna d_c(x,y) \quad \text{ for all } x,y \in G.$$
\end{lem}

In the following definition we look more analytically at $V_1$, such that we can define certain possible properties of subspaces of $V_1$. Later the presence of these properties will be very useful.

\begin{defi}
Let $G$ be a stratified nilpotent Lie group of dimension $n$ with Lie algebra $\mf g$. Let $\mc H$ be the horizontal distribution induced by the first layer $V_1$. Further let $n_1=\dim V_1$. Then this distribution can be described as the set of common zeros of a set of $1$-forms $\{\eta_1,...,\eta_{n-n_1}\}$. These forms induce a (vector-valued) form 
$$\Omega=(\omega_1,...,\omega_{n-n_1}): \Lambda^2 V_1 \to \mf g/V_1 \cong\mathbb R^{n-n_1},$$
the \textup{curvature form}, where the $\omega_i$ denote the differentials $\omega_i\defeq d\eta_i $ for \\$ i=1,...,n-n_1$. \\
Let $(\sigma_{ij}) \in \RR^{(n-n_1)\times k}$. For a $k$-dimensional subspace $S\subset V_1$ consider the system of equations 
$$\omega_i(\xi,X_j)=\sigma_{ij} \quad i=1,...,n-n_1 \text{ and } j=1,...,k$$ 
where $\{X_j\}$ is a basis of $S$.\\
 Then $S$ is called \textup{$\Omega$-regular}, if for any $(\sigma_{ij}) \in \RR^{(n-n_1)\times k}$ the system of equations has a solution $\xi \in V_1$.\\
Further a subspace $S\subset V_1$ is called \textup{$\Omega$-isotropic}, if $\Omega$ restricted to $\Lambda^2S$ is the zero form.
\end{defi}

Let $b_1,...,b_n$ be a basis of the Lie algebra $\mf g$, such that $b_1,...,b_{n_1}$ span the first layer $V_1$. Then one can choose the $1$-forms $\eta_i$ as the dual forms of the remaining basis vectors $b_{n_1+1},...,b_n$:
$$ \eta_i=b_{n_1+i}^* \qquad i \in \{1,...,n-n_1\} \ .$$
Using the formula
$$(p+1)!(\opna d \gamma)(X_0,...,X_p)=\sum_{i<j}(-1)^{i+j+1}\gamma([X_i,X_j],X_0,...,\hat X_i,...,\hat X_j,...,X_p)$$
for the differential of a left-invariant $p$-form $\gamma$, one gets
$$\omega_i(X_0,X_1)=\frac{1}{2} \cdot b_{n_1+i}^*([X_0,X_1])\ .$$
So we see, that an $\Omega$-isotropic subspace $S$ is nothing else than an abelian Lie sub-algebra, which is totally contained in the first layer $V_1$ of the Lie algebra.\\
Further we can interpret the property ``$\Omega$-regular'', as something like \textit{in general position}.

We close this section mentioning two important properties of Carnot groups:

\begin{prop}[{see \cite{Pansu}}]
Let $G$ be a stratified nilpotent Lie group equipped with a left-invariant Riemannian metric $g$ with associated length metric $\opna d_g$. Then the metric spaces $(G, \frac{1}{r}\opna d_g,e)$ converge in the pointed Gromov-Hausdorff sense for $r \to \infty$ to $(G,\opna d_c,e)$, where $\opna d_c$ denotes the Carnot-Carath\'eodory metric.
\end{prop}

This means, that the group $G$ equipped with its Carnot-Carath\'eodory metric, is the (up to isometry) unique asymptotic cone of $(G,\opna d_g)$.

It can be shown (see \cite[Theorem 2]{Mitchell}), that the Hausdorff-dimension $D$ of $(G,\opna d_c)$ is given by 
$$D = \sum_{j=1}^d j \cdot \dim V_j$$
where $d$ is the degree nilpotency of $G$ and the $V_j$ are the summands of the grading of the Lie algebra $\mf g$. We will see this number again, when we establish the equivalence classes of high-dimensional filling functions of stratified nilpotent Lie groups.


\section{Filling functions of stratified nilpotent Lie groups}
\vspace*{-5mm}
We start with our results for the filling functions of stratified nilpotent Lie groups.

\vspace*{-5mm}
\subsection{Results}
\vspace*{-5mm}
We will prove, that the existence of a $(k+1)$-dimensional, $\Omega$-regular abelian subalgebra in the first layer of the Lie algebra leads to Euclidean filling functions up to dimension $k+1$. We additionally have to assume, for technical reasons,  the existence of a scalable lattice in $G$. We will see, that in the case of $2$-step nilpotent Lie groups this assumption can be dropped. \\
Remember the notation $s_t:G \to G\ , \ t>0,$ for the scaling automorphisms of a stratified nilpotent Lie group (see Section \ref{SectionSNLG}).

\begin{Thm}[{Euclidean filling functions}]\label{Thm1}
Let $G$ be a stratified nilpotent Lie group equipped with a left-invariant Riemannian metric. Further let $\mf g$ be the Lie algebra of $G$  and $V_1$ be the first layer of the grading and let $k \in \mathbb N$. If there exists a  lattice $\Gamma\subset G$ with $s_2(\Gamma)\subset\Gamma$ and a $(k+1)$-dimensional $\Omega$-isotropic, $\Omega$-regular subspace $S\subset V_1$,  then holds:
$$F^{j+1}_G(l) \sim l^\frac{j+1}{j} \quad \text{for all } j \le k.$$
Let $d$ denote the degree of nilpotency of $G$. Then further holds \
$F^{k+2}_G(l) \preccurlyeq l^{\frac{k+1+d}{k+1}}$.

\end{Thm}

The upper bound on $F^{k+2}_G$ is the (higher dimensional) analogue to Gromov's bound \ $\delta_\Gamma(n) \preccurlyeq n^{1+d}$ \ on the ($1$-dimensional) Dehn function for nilpotent groups in \textup{\cite[5.A$'$.5]{GGT}} (see also \textup{\cite{Pittet95}}) as there is the relation $\delta^{k+1}\preccurlyeq F^{k+2}$ for $k\ge 1$.

\begin{rem}
Gromov proved in \textup{\cite{GGT}} the following dimension-formula for $\Omega$-isotropic, $\Omega$-regular subspaces $S\subset V_1$: 
$$\dim V_1 - \dim S \ge \dim S (\dim \mf g -\dim V_1) \quad (*)$$
For the existence of $\Omega$-isotropic, $\Omega$-regular subspaces this means, that the horizontal distribution has to be large, i.e. $\dim V_1 >> \opna{codim}_\mf g V_1$. This formula comes from the fact, that the $\Omega$-isotropy and $\Omega$-regularity of $S$ implies that the linear map 
$$\Omega_\bullet :V_1 \to \opna{Hom}(S,\mf g / V_1), X \mapsto \Omega(X, \_)$$
is surjective and vanishes on $S$. The left hand side in the above inequality equals the dimension of $V_1 / S$ and the right hand side the dimension of $\opna{Hom}(S,\mf g / V_1)$. As $S$ is in the kernel of $\Omega_\bullet$ we get by the surjectivity of $\Omega_\bullet$ the inequality as necessary condition. \\
And on the other hand Gromov proved, that $(*)$ is sufficient for  \textup{generic} $\Omega$, i.e. for a class of forms, which form an open and everywhere dense subset.
\end{rem}

Our second result refers to ``high dimensions''. Again we assume the existence of a scalable lattice and a $(k+1)$-dimensional abelian $\Omega$-regular subalgebra in the first layer of the Lie algebra.. Then we can prove sub-Euclidean filling functions in the $k$ dimensions below the dimension of the group (if the group is not abelian). So the geometry of stratified nilpotent Lie groups behaves not Euclidean in high dimensions.

\begin{Thm}[{Sub-Euclidean filling functions}]\label{Thm2}
Let $G$ be an $n$-dimensional stratified nilpotent Lie group equipped with a left-invariant Riemannian metric. Further let $\mf g$ be the Lie algebra of $G$  with grading $\mf g=V_1\oplus ... \oplus V_d$. Denote by $D =  \displaystyle{\sum_{i=1}^d i \cdot \dim V_i}$ \  the Hausdorff-dimension of the asymptotic cone of $G$ and let $k\in \NN$. If there exists a  lattice $\Gamma\subset G$ with $s_2(\Gamma)\subset\Gamma$ and  a $(k+1)$-dimensional $\Omega$-regular, $\Omega$-isotropic subspace $S \subset V_1$, then holds:
$$F^{n-j}_G(l) \sim l^\frac{D-j}{D -j-1} \quad \text{for all } j \le k-1.$$
\end{Thm} 

Note that every lattice $\Gamma$ in a nilpotent Lie group $G$ is cocompact. So $G$ and $\Gamma$ are quasi-isometric. Therefore their asymptotic cones are isometric and have coinciding Hausdorff-dimensions.\\
Therefore the above result extends a classical result of Nicolas Varopoulos (see \cite{GromovR},\cite{Varopoulos}).

\begin{Th}[{compare  \cite{Varopoulos}}]
Let $\Gamma$ be a lattice in an $n$-dimensional nilpotent Lie group $G$. Further denote by $D$ the Hausdorff dimension of the asymptotic cone of $\Gamma$. Then holds:
$$\delta_\Gamma^{n-1}(l) \sim l^\frac{D}{D -1} \ .$$
\end{Th}

Varopoulos' result corresponds to the case $j=0$ in Theorem \ref{Thm2}.

Now we turn to the special case of simply connected $2$-step nilpotent Lie groups. As seen in Section \ref{SectionSNLG}, all simply connected $2$-step nilpotent Lie groups are stratified nilpotent Lie groups and so fit to our situation. We will see that in every such group there is a lattice $\Gamma$ which satisfies the condition $s_2(\Gamma)\subset \Gamma$. So this doesn't remain a restriction to the Lie group and we can drop this requirement. \\ This leads to the following version:

\begin{Thm}[{Filling functions of $2$-step nilpotent Lie groups}]\label{Thm3}
Let $G$ be an $n$-dimensional simply connected $2$-step nilpotent Lie group equipped with a left-invariant Riemannian metric. Further let $\mf g$ be the Lie algebra of $G$ with grading $\mf g= V_1 \oplus V_2$. Let $n_2=\dim V_2$ and let $k \in \mathbb N$. If there exists a   $(k+1)$-dimensional $\Omega$-regular, $\Omega$-isotropic subspace $S\subset V_1$,  then holds:
\begin{enumerate}[(i)]
\item $F^{j+1}_G(l) \sim l^\frac{j+1}{j} \qquad$ for all $j \le k$,

\item $F^{k+2}_G(l) \preccurlyeq l^{\frac{k+3}{k+1}}$ \ ,

\item $F^{n-j}_G(l) \sim l^\frac{n+n_2-j}{n+n_2-j-1}\quad$ for all $j \le k-1$.

\end{enumerate}
\end{Thm}

The above theorems state a super-Euclidean upper bound for the filling function in the first dimension above the dimension of the $\Omega$-regular, $\Omega$-isotropic subspace. In some cases we can get in this dimension a super-Euclidean lower bound, too.

\begin{Thm}[{A super-Euclidean filling function}]\label{Thm4}
Let $G$ be a stratified nilpotent Lie group equipped with a left-invariant Riemannian metric. Further let $\mf g$ be the Lie algebra of $G$  with grading $\mf g=V_1\oplus ... \oplus V_d$. Let $k_0,k_1\in \NN$, such that $(k_0+1)$ is the maximal dimension of an $\Omega$-regular, $\Omega$-isotropic subspace of $V_1$ and $(k_1+1)$ is the maximal dimension of an  $\Omega$-isotropic subspace of $V_1$. Further let one of the following two conditions be satisfied:
\begin{enumerate}[a)]
\item There is an $k_0 \le k \le k_1$ such that there is an integral current $T\in \boldsymbol I^{cpt}_{k+1}(G,\opna d_c)$ with $\partial T=0$ and $T\ne 0$ but no integral current $S \in \boldsymbol I^{cpt}_{k+2}(G,\opna d_c)$ with $\partial S=T$.
\item The two numbers $k_0$ and $k_1$ coincide: $\quad k_0=k_1 \eqdef k.$
\end{enumerate}
Then holds:
$$F^{k+2}_G(l) \succnsim l^\frac{k+2}{k+1} \ .$$
\end{Thm}

We will see in the proof, that condition \textit{b)} implies condition \textit{a)}. Nevertheless we allow condition \textit{b)} to stand in the theorem, as it is easier to check directly for some specific Lie groups.

A stratified nilpotent Lie group which fulfils the conditions of Theorem \ref{Thm4} has at least one super-Euclidean filling function.
This wouldn't be possible, if the group would be a space of non-positive curvature. Of course a stratified nilpotent Lie group can't be a space of strictly positive curvature, as it is diffeomorphic to some $\RR^{N}$ and therefore not bounded. So the above theorem recovers for (some) stratified nilpotent Lie groups the result of Wolf \cite{Wolf} about the appearance of all different types of sectional curvature.
We will see, that the octonionic Heisenberg Groups $H^n_\OO$ are such groups. The Heisenberg Groups over $\CC$ and $\HH$ satisfy the conditions of Theorem \ref{Thm4}, too, but the super-Euclidean growth-type of their filling functions in dimension $n+1$ is already known exactly (see \cite{Young1} and \cite{Gruber2}).

%
%

\subsection{Tools for the proofs}\label{S3}
\vspace*{-5mm}
For the proofs of the bounds on the filling functions we will use the following theorems. The first of them, due to Burillo, will be crucial to establish lower bounds on the filling functions.

\begin{thm}[{see \cite[Prop. 1.3] {Burillo}}]\label{Burillo} 
Let $G$ be a stratified nilpotent Lie group equipped with a left-invariant Riemannian metric and let $m \in \NN$. If there exists a Lipschitz $m$-chain $b$ and a closed $G$-invariant $m$-form $\gamma$ in G and constants $C,r,s >0$ such that \vspace{-3mm}
\begin{enumerate}[1)]
\item $\mass(s_t(\partial b))\le Ct^r$\ ,			\vspace{-3mm}
\item $\int_b \gamma >0$\ ,					\vspace{-3mm}
\item $s_t^*\gamma = t^s\gamma$\ ,			\vspace{-3mm}
\end{enumerate}
then holds \ $F^m_G(l) \succcurlyeq l^\frac{s}{r}$.
\end{thm}

The following two theorems, both due to Young, are essential to establish upper bounds on the filling functions. For these theorems we need the notion of horizontal maps introduced in Section \ref{SectionSNLG}.
\begin{thm}[{see \cite[Thm. 3]{Young1}}]\label{ThmYoung}
Let $G$ be a stratified nilpotent Lie group equipped with a left-invariant Riemannian metric, let $(\tau,f)$ be a triangulation of $G$ and let $\phi:\tau \to G$ be a $m$-horizontal map in bounded distance to $f$. Further let $(\eta,h)$ be a triangulation of $G\times [1,2]$ which restricts on $G\times\{1\}$ to $(\tau,f)$ and on $G\times \{2\}$ to $(\tau, s_2\circ f)$ and let $\psi:\eta \to G$ be an $m$-horizontal map which extends $\phi$ and $s_2\circ \phi$ (i.e. $\psi^{(m)}_{\mid h^{-1}(G\times \{1\})}=\phi^{(m)}$ and $\psi^{(m)}_{\mid h^{-1}(G\times \{2\})}=s_2\circ \phi^{(m)}$). Then holds:
$$F^{j+1}_G(l)\preccurlyeq l^\frac{j+1}{j} \quad \text{for all } j \le m-1.$$
\end{thm}

\begin{thm}[{see \cite[Prop. 8]{YoungII}}]\label{ThmYoungII} 
Let $G$ be a stratified nilpotent Lie group equipped with a left-invariant Riemannian metric  and let $\Gamma\subset G$ be a lattice with $s_2(\Gamma)\subset \Gamma$ and let $(\tau, f)$ be a $\Gamma$-adapted  triangulation, i.e. $f$ is $\Gamma$-equivariant. Further let $(\widetilde \tau, \widetilde f)$ be a $s_2(\Gamma)$-adapted  triangulation of $G\times [1,2]$, such that for $i \in \{1,2\}$ the restriction to $G\times \{i\}$ coincides with the triangulation $(\tau, s_i\circ f)$. Denote by $D$ the Hausdorff-dimension of the asymptotic cone $(G,d_c)$ of $G$. If there is a $s_2(\Gamma)$-equivariant, $m$-horizontal, piecewise smooth map $\psi: G\times [1,2] \cong \widetilde \tau \to G$ with $\psi(g,2)=s_2(\psi(s_\frac{1}{2}(g),1))$, then holds:
$$F^{n-j}_G(l) \preccurlyeq l^\frac{ D-j}{D -j-1} \quad \text{for all } j\le m-1.$$
\end{thm}

To check that our conditions imply the conditions of these theorems, we will use the $h$-principle and microflexibility. Roughly speaking, the existence of the $\Omega$-regular, $\Omega$-isotropic subspace will give us small horizontal submanifolds, which we are able to agglutinate to the desired triangulation. In fact we use the following proposition due to Gromov:

\begin{prop}[{see \cite[4.4 Corollary]{Gromov}} or Corollary \ref{LemGr2}]\label{ThmGromov}
Let $G$ be a stratified nilpotent Lie group with Lie algebra $\mf g$ and $m \in \mathbb N$. Further let $S$ be a $m$-dimensional $\Omega$-isotropic, $\Omega$-regular horizontal subspace of $\mf g$. Then every continuous map $f_0: T \to G$ from an $m$-dimensional simplicial complex $T$ into the stratified nilpotent Lie group $G$ can be approximated by continuous, piecewise smooth, piecewise horizontal maps $f:T \to G$.
\end{prop}

As Gromov only sketches the proof of this proposition, we give a detailed proof in Section \ref{S4} (see Corollary \ref{LemGr2}).

%
%

\subsection{Proofs for Filling Functions}\label{S5}
\vspace*{-5mm}
In this chapter we give the proofs for the bounds on the filling functions of stratified nilpotent Lie groups.
We start with the proof of Theorem \ref{Thm1}.

\vspace*{-5mm}
\subsubsection{The proof of Theorem \ref{Thm1}}
\vspace*{-5mm}

To prove  Theorem \ref{Thm1} we split its statement into three parts: The upper bounds in the dimensions from dimension $2$ up to dimension $k+1$, the lower bounds in these dimensions and the upper bound in dimension $k+2$. We prove the first part in Proposition \ref{Propcon}, the second part in Proposition \ref{Propeta} and finally the third part in Proposition \ref{dimplus}.

\begin{prop}\label{Propcon}
Let $G$ be a stratified nilpotent Lie group equipped with aleft-invariant Riemannian metric. Let $\mf g$ be Lie algebra of $G$ and $V_1$ as first layer of the grading of the Lie algebra. Further let $k \in \mathbb N$ and let $d$ be the degree of nilpotency of $G$. If there exists a lattice $\Gamma$ with $s_2(\Gamma) \subset \Gamma$ and a $(k+1)$-dimensional $\Omega$-isotropic, $\Omega$-regular subspace $S\subset V_1$, then holds:
$$F^{j+1}_G(l) \preccurlyeq l^\frac{j+1}{j} \quad \text{for all } j \le k.$$
\end{prop}

\begin{proof}
First we look at the quotient $M=G/\Gamma$. As $G$ is a simply connected nilpotent Lie group and as $\Gamma$ is a lattice in $G$, we know by \cite[Theorem 2.18]{Raghunathan} that $\Gamma$ is torsion free. So $M$ is a smooth manifold.\\
Let $(\tau_M,f_M)$ be a triangulation of $M$. This triangulation $(\tau_M,f_M)$ lifts to a $\Gamma$-invariant triangulation $(\tau,f)$ of $G$.\\
By \cite[Lemma 4.5]{Young1} we get an $s_2(\Gamma)$-invariant triangulation $(\eta,\tilde f)$ of $G\times[1,2]$ which restricts on $G\times\{1\}$ to $(\tau,f)$ and on $G\times \{2\}$ to $(\tau, s_2\circ f)$. Here, the $s_2(\Gamma)$-action on $G\times \{1,2\}$ is defined by $\varphi_{\gamma}(g,i)=(s_{\frac{1}{3-i}}(\gamma)g,i)$ for $(g,i)\in G \times \{1,2\}$ and $\gamma \in s_2(\Gamma)$.\\
Define the map $\psi_0: \eta \to G$ by $\psi_0=\opna{pr}_G \circ \tilde f$ where $\opna{pr}_G: G\times [1,2] \to G$ denotes the projection to the first factor. Let further $f^{(k+1)}:\tau^{(k+1)}\to G$ and $\psi_0^{(k+1)}:\eta^{(k+1)} \to G$ be the restrictions of $f$ and $\psi_0$ to the $(k+1)$-skeletons of $\tau$ and $\eta$.
\\
As there is a $(k+1)$-dimensional $\Omega$-isotropic, $\Omega$-regular subspace $S\subset V_1$, the group $G$ fulfils the conditions of Proposition \ref{ThmGromov}) for $m=k+1$. Therefore (mentioning the remark following Proposition \ref{ThmGromov}) we can approximate $f^{(k+1)}$ by a horizontal map 
$$\phi^{(k+1)}: \tau^{(k+1)} \to G \ .$$
Further, as $\psi_0$ extends $f$ and $s_2\circ f$, i.e. 
$${\psi_0}_{\mid \tilde f^{-1}(G\times \{1\})}=f \quad \text{ and }\quad {\psi_0}_{\mid \tilde f^{-1}(G\times \{2\})}=s_2\circ f$$
we can, using again Proposition \ref{ThmGromov}), approximate $\psi_0^{(k+1)}$ by a horizontal map 
$$\psi^{(k+1)}:\eta \to G$$
which extends $\phi^{(k+1)}$ and $s_2\circ \phi^{(k+1)}$, i.e. 
$$\psi^{(k+1)}_{\mid \tilde f^{-1}(G\times \{1\})}=\phi^{(k+1)}\quad \text{ and }\quad\psi^{(k+1)}_{\mid \tilde f^{-1}(G\times \{2\})}=s_2\circ \phi^{(k+1)} \ .$$
Now we extend $\phi^{(k+1)}$ and $\psi^{(k+1)}$ to the whole simplicial complexes such that $\psi:\eta \to G$ extends $\phi:\tau \to G$ and $s_2\circ \phi:\tau \to G$. We do this by filling successively the boundary of each $r$-simplex $\widetilde{\Delta^r}$ of $\eta$, $r \ge k+2$, by a Lipschitz map $\psi^{(r)}:\widetilde{\Delta^r} \to G$ with $\psi^{(r)}(\partial\widetilde{\Delta^r})=\psi^{(r-1)}(\partial\widetilde{\Delta^r})$. This can be done as $G$ is contractible.\\
So we have the triangulations $(\tau,f)$ and $(\eta,\tilde f)$ and the $(k+1)$-horizontal maps $\phi$ and $\psi$ in bounded distance to $f$ and $\tilde f$ as required in Theorem \ref{ThmYoung}. \\
So $G$ fulfils the conditions for Young's filling theorem and we get the bound
$F^{j+1}_G(l) \preccurlyeq l^\frac{j+1}{j} \text{ for all } j \le k$.
\end{proof}

To prove the remaining lower bounds in Theorem \ref{Thm1} we use  Theorem \ref{Burillo} of Burillo:

We have to construct a $(j+1)$-form $\gamma$ and a closed $(j+1)$-chain $b$ for the constants $r=j$ and $s=j+1$ for $1 \le j \le  k$. 
To do this we use again the results of the previous chapter, in particular the local integrability, i.e. the existence of germs of horizontal submanifolds.

\begin{prop}\label{Propeta}
Let $G$ be a stratified nilpotent Lie group equipped with a left-invariant Riemannian metric. Let $\mf g$ be the Lie algebra of $G$ and $V_1$ be the first layer of the grading of the Lie algebra and let $k\in \NN$. If there exists a $(k+1)$-dimensional $\Omega$-isotropic, $\Omega$-regular subspace $S\subset V_1$ then holds:
$$F^{j+1}_G(l) \succcurlyeq l^\frac{j+1}{j}\quad \text{for all } j \le k.$$
\end{prop}

\begin{proof} 
Let $1 \le j \le k$.
We will show that there exist $\gamma$ and $b$ which fulfil the conditions of Burillo's filling theorem  for $r=j$ and $s=j+1$ (Theorem \ref{Burillo}).\\
Let $X_1,...,X_{k+1}$ be a basis of $S$ and define $S_j=\langle X_1,...,X_{j+1}\rangle$.
By  Lemma \ref{localhlem}  there is an integral submanifold $M$ to $S_j\subset S$, i.e. $T_pM=dL_pS_j \quad \forall p \in M$, where $dL_g$ denotes the differential of the left-multiplication by $g\in G$. Let $\varepsilon >0$ and $b=B_\varepsilon^M(id)$ be the $\varepsilon$-ball in $M$.  Further let $\gamma=X_1^*\wedge X_2^*\wedge ... \wedge X_{j+1}^*$, where $X_i^*$ denotes the $G$-invariant dual-form to $X_i$ defined by $X_u^*(X_v)=\delta_{uv}$. Then $\gamma$ is a closed $G$-invariant $(j+1)$-form as all of the $X_i$ lie in $S_j \subset S\subset V_1$ which has trivial intersection with $[\mf g, \mf g]$. \\
It remains to check the conditions of Theorem \ref{Burillo}:\vspace*{-3mm}
\begin{enumerate}[]
\item \textit{1)} $\mass(s_t(\partial b))=\mass(\partial b)\cdot t^{j}$ \ as $\partial b \subset M$ is a horizontal $j$-cycle. \\ \hspace*{5mm}As constant we can choose $C=\mass(\partial b)$.																	\vspace*{-3mm}
\item \textit{2)} $\int_b \gamma >0$ \ as $\gamma$ is a multiple of the volume form of $M$. 						\vspace*{-3mm}
\item \textit{3)} $s_t^*\gamma = t^{j+1} \gamma$ \ as all $X_i$ lie in $V_1$ and $\gamma$ is a $(j+1)$-form. 		\vspace*{-3mm}
\end{enumerate}
So the conditions of Theorem \ref{Burillo} are fulfilled for $r=j$ and $s=j+1$. Therefore holds for all $j\le k$: \ $F^{j+1}_G(l) \succcurlyeq l^\frac{j+1}{j}$.
\end{proof}

It now remains to prove the upper bound in the dimension above the dimension of the $\Omega$-regular $\Omega$-isotropic subspace $S$:

\begin{prop}\label{dimplus}
Let $G$ be a stratified $d$-step nilpotent Lie group equipped with a left-invariant \mbox{Riemannian} metric. Let $\mf g$ be the Lie algebra of $G$ and $V_1$ be the first layer of the grading of the Lie algebra and let $k\in \mathbb N$. If there exists a lattice $\Gamma$ with $s_2(\Gamma) \subset \Gamma$ and a $(k+1)$-dimensional $\Omega$-isotropic, $\Omega$-regular subspace $S\subset V_1$, then holds:
$$F^{k+2}_G(l) \preccurlyeq l^\frac{k+1+d}{k+1} \ .$$
\end{prop}

\begin{proof}
With the proof of Proposition \ref{Propcon} we have the triangulations $(\tau,f)$, $(\eta, \tilde f)$ and the $(k+1)$-horizontal maps $\phi$, $\psi$ in bounded distance to $f$ anf $\tilde f$. Therefore the conditions of \cite[Theorem 7]{Young1} are fulfilled for a function $\sim t^{k+1+d}$ (compare also \cite[Discussion before Theorem 8]{Young1}). \\
Let $\Delta=\Delta^{k+2}$ be a $(k+2)$-simplex of $\tau$. Then $\phi$ is already horizontal on $\partial \Delta$. Identify $\Delta$ with the cone $\opna{Cone}(\partial \Delta)=(\partial \Delta \times [0,1])/(\partial \Delta \times \{0\})$ over the boundary. Define the map 
$$h_\Delta: \Delta \to G, \overline{(x,a)} \mapsto s_a(\phi(x))\ .$$
Then $h_\Delta$ coincides with $\phi$ on $\partial \Delta$. Further holds
$$\mass((s_t\circ h_\Delta)(\Delta)) \le t^{k+1+d} \cdot \mass(h_\Delta(\Delta))$$
as in every point the tangent space to $h_\Delta$ is the span of a $(k+1)$-dimensional horizontal subspace and another vector $v$. As $G$ is $d$-step nilpotent, for $v$ holds $\|s_t(v)\|\le t^d\|v\|$. Replace $\phi$ by $\phi'_{|\Delta}=h_\Delta$ for every $(k+2)$-simplex $\Delta$ of $\tau$ and extend $\phi'$ to all of $\tau$ (compare with the proof of Proposition \ref{Propcon}). \\
With the same techniques one can construct a sufficient map $\psi'$. \\
By \cite[Theorem 7]{Young1} follows: $F^{k+2}_G(l) \preccurlyeq l^\frac{k+1+d}{k+1}$.
\end{proof}

So we obtain the lower bounds on the filling functions by Proposition \ref{Propeta} and Proposition \ref{dimplus}. Together with the upper bounds from Proposition \ref{Propcon} this proves Theorem \ref{Thm1}.

%
%

\subsubsection{The proof of Theorem \ref{Thm2}}
\vspace*{-5mm}

Similarly as for the proof of Theorem \ref{Thm1}, we split the statement of \mbox{Theorem \ref{Thm2}} into parts. In Proposition \ref{topup} we prove the upper bounds and in Proposition \ref{toplow} we prove the lower bounds on the filling functions.

\begin{prop}\label{topup}
Let $G$ be an $n$-dimensional stratified nilpotent Lie group equipped with a left-invariant Riemannian metric. Let $\mf g$ be the Lie algebra of $G$ with grading \\$\mf g=V_1\oplus ... \oplus V_d$. Further let $D =  \displaystyle{\sum_{i=1}^d i \cdot \dim V_i}$ be the Hausdorff-dimension of the asymptotic cone of $G$ and let $k\in \NN$. If there exists a lattice $\Gamma \subset G$ with $s_2(\Gamma)\subset \Gamma$ and a  $(k+1)$-dimensional $\Omega$-regular, $\Omega$-isotropic subspace $S \subset V_1 $, then holds:  
$$F^{n-j}_G(l) \preccurlyeq l^\frac{D-j}{D -j-1} \quad \text{for all } j \le k-1.$$
\end{prop}

\begin{proof}
We want to use Young's filling theorem for high dimensions (Theorem \ref{ThmYoungII}). So we have to check that our conditions imply the conditions of this theorem. This means we have to construct an $s_2(\Gamma)$-adapted triangulation $(\tilde \tau, \tilde f)$ of $G\times [1,2]$, such  that the restrictions to $G\times \{i\}$, $i=1,2$ are transformed into each other by the scaling $s_2$. Further we need a piecewise smooth, $s_2(\Gamma)$-equivariant, $k$-horizontal map $\psi: \tilde \tau \to G$ with $\psi(x,2)=s_2(\psi(s_\frac{1}{2}(x),1))$. Here we used the notation $(x,t)$ for points in $\tilde\tau$ as this simplicial complex is homeomorhic to $G\times [1,2]$. Further does $s_\frac{1}{2}$ denote the change from the triangulation of $G\times \{2\}$ to the triangulation of $G\times \{1\}$.\\\quad\\
We look at the quotient $M=G/\Gamma$. As $G$ is a simply connected nilpotent Lie group and as $\Gamma$ is a lattice in $G$, we know by \cite[Theorem 2.18]{Raghunathan} that $\Gamma$ is torsion free. So $M$ is a smooth manifold.\\
Let $(\tau_M,f_M)$ be a triangulation of $M$. Then $(\tau_M,f_M)$ lifts to a $\Gamma$-adapted triangulation $(\tau,f)$ of $G$.\\
Let $M_2=G/s_2(\Gamma)$. This is again a smooth manifold. The $\Gamma$-adapted triangulation $(\tau,f)$ projects down to a triangulation $(\tau_{M_2}, f_{M_2})$ of $M_2$.\\
Denote by $s_i(\tau)$ the triangulation $(\tau, s_i \circ f)$. Let $\varphi_\gamma: \tau \to \tau $, $\gamma \in \Gamma$, be the $\Gamma$-action on $\tau$. Then we define the $s_2(\Gamma)$-action on $s_1(\tau)$ by $\varphi^1_\gamma=\varphi_{s_\frac{1}{2}(\gamma)}$ and on $s_2(\tau)$ by $\varphi^2_\gamma=\varphi_{\gamma}$ for $\gamma \in s_2(\Gamma)$. In respect to this actions, both triangulations are $s_2(\Gamma)$-adapted.\\
As in \cite[Proposition 4.5]{Young1} we can extend the projected triangulation to a triangulation $(\tilde \tau_{M_2}, \tilde f_{M_2})$ of $M_2 \times [1,2]$, such that this restricts to $s_i(\tau)/s_2(\Gamma)$ on $M_2\times\{i\}$, $i=1,2$.  Then $(\tilde \tau_{M_2}, \tilde f_{M_2})$ lifts to a $s_2(\Gamma)$-adapted triangulation $(\tilde \tau ,\tilde f)$ of $G\times [1,2]$ with the required restrictions.\\\quad\\
To construct the map $\psi$ we use the $h$-principle on $M_2$. The translates of the $(k+1)$-dimensional $\Omega$-regular, $\Omega$-isotropic subspace $S\subset V_1$ descend to a continuous subbundle of the tangent bundle of $M_2$ (consisting of $(k+1)$-planes). As the property to be microflexible is local, it descends to $M_2$, too. So we can use Proposition \ref{approx} (see Remark at the end of Section \ref{S4}).\\
As $\tilde \tau_{M_2}$ is homeomorphic to $M_2 \times [1,2]$, we can write each point of $\tilde \tau_{M_2}$ as $(x,t)$, where the second entry is the image of the point in the second factor of $M_2 \times [1,2]$.\\
Denote by $\opna{pr}_{M_2}$ the projection to the first factor of $M_2 \times [1,2]$. Define the map
$$\psi_0 : \{(x,1) \in \tilde \tau_{M_2}\} \to M_2 \ , \ (x,1) \mapsto \opna{pr}_{M_2}\big( \tilde f_{M_2}(x,1) \big)=\opna{pr}_{M_2}\big( f(x)s_2(\Gamma)\big) \ .$$
We approximate $\psi_0$ on $\{(x,1) \in \tilde \tau_{M_2}\}$ by a $(k+1)$-horizontal immersion $\psi^1_{M_2}$ using Proposition \ref{approx}. Then we lift $\psi^1_{M_2}$ to a $s_2(\Gamma)$-equivariant map $\psi': \tau \to G$ and define
$$\psi^2_{M_2}:\{(x,2) \in \tilde \tau_{M_2}\} \to M_2 \ , \ (x,2) \mapsto \opna{pr}_{M_2}\big( (s_2 \circ \psi')(s_\frac{1}{2}(x))s_2(\Gamma)\big) $$
where $s_\frac{1}{2}(x)$ denotes the change from the triangulation $s_2(\tau)$ to the triangulation $s_1(\tau)$.\\
Then we extend $\psi^1_{M_2}$ and $\psi^2_{M_2}$ to a Lipschitz map $\psi'_{M_2}$ from $\tilde \tau_{M_2}$ to $M_2$ by defining ${\psi'_{M_2}}_{|\tilde \tau_{M_2}^{(0)}}={\psi_{0}}_{|\tilde \tau_{M_2}^{(0)}}$ on the vertices and then successively filling the boundaries of the simplices.\\
To complete the construction we use (the relative version of) the Holonomic Approximation Theorem from \cite{hprinciple} and approximate $\psi'_{M_2}$ on all of $\tilde \tau_{M_2}$, fixed on $\{(x,i)\in \tilde \tau_{M_2} \mid i \in \{1,2\}\}$, by a $k$-horizontal, piecewise smooth map 
$$\psi_{M_2}:\tilde \tau_{M_2} \to M_2$$
with $\psi_{M_2}(x,i)=\psi'_{M_2}(x,i)$ for all $(x,i) \in \tilde \tau_{M_2}$ with $i \in \{1,2\}$.\\
This map rises to a $s_2(\Gamma)$-equivariant, piecewise smooth, $k$-horizontal map \\$\psi :\tilde \tau \to G$ with $\psi(x,2)=s_2(\psi(s_\frac{1}{2}(x),1))$.\\\quad\\
So the conditions of Theorem \ref{ThmYoungII} are fulfilled and we get $F^{n-j}_G(l) \preccurlyeq l^\frac{D-j}{D -j-1}$ for all $j \le k-1$. 
\end{proof}

It remains to prove the lower bounds. As in the proof of the low dimensions (Theorem \ref{Thm1}), we will use Burillo's  Theorem \ref{Burillo}.

\begin{prop}\label{toplow}
Let $G$ be an $n$-dimensional stratified nilpotent Lie group equipped with a left-invariant Riemannian metric. Let $\mf g$ be the Lie algebra of $G$ with grading \\$\mf g=V_1\oplus ... \oplus V_d$. Further let $D =  \displaystyle{\sum_{i=1}^d i \cdot \dim V_i}$ be the Hausdorff-dimension of the asymptotic cone of $G$ and let $k \in \NN$. If there exists  a $(k+1)$-dimensional $\Omega$-isotropic subspace $S\subset V_1$, then holds:
$$F^{n-j}_G(l) \succcurlyeq l^\frac{D-j}{D -j-1} \quad \text{for all } j \le k.$$
\end{prop} 

\begin{proof}
As all left-invariant Riemannian metrics on $G$ are biLipschitz equivalent, we can choose one of our liking. So let the left-invariant Riemannian metric on $G$  fitting to the grading of $\mf g$, i.e.  $V_s\perp V_t$ for $s \ne t$.\\
Let $j\le k$. We will use the filling theorem of Burillo (see Theorem \ref{Burillo}) with $m=n-j$.\\
To do this, we consider the grading $\mf g=V_1\oplus ... \oplus V_d$. We choose an orthonormal basis $B_1=\{v_1^{(1)},...,v_{\dim V_1}^{(1)}\}$ of $V_1$, such that  the vectors $v_1^{(1)},...,v_{k+1}^{(1)}$ span the $(k+1)$-dimensional $\Omega$-regular, $\Omega$-isotropic subspace $S\subset V_1$. Then we choose on each summand $V_i$ an orthonormal basis 
$$B_i=\{v_1^{(i)},...,v_{\dim V_i}^{(i)}\}$$
which gives us an orthonormal basis $B=\bigcup_{i=1}^d B_i$ of the Lie algebra $\mf g$.\\
As $m$-chain $b$ we choose now the image under the exponential map of the unit cube in all coordinates of $B$ except $v_1^{(1)},...,v_{j}^{(1)}$, i.e.
$$b=\exp(\{\sum_{i=1}^d\sum_{q=1}^{\dim V_i} \alpha_{i,q}\cdot v_q^{(i)}\mid 0\le \alpha_{i,q} \le 1 \ \forall i\ \forall q \text{, and } \alpha_{1,q}=0 \text{ for } 1 \le q \le j\})\ .$$
The vectors of the $i^{th}$ layer $V_i$ of the Lie algebra are scaled under the scaling automorphism $L(s_t):\mf g \to \mf g$  in the way 
$$L(s_t)(v^{(i)})=t^iv^{(i)}\ .$$
So we have for the cube $b$
$$\mass(s_t(b))=t^{D-j} \mass(b)\ .$$
The proof of this scaling behaviour follows the same lines as the proof of Lemma \ref{masslem}. The only difference is, that the tangent space of the image of the chain $b$ always is a translate of $\mf g /W$ for $W=\langle v_1^{(1)},v_2^{(1)},...,v_{j}^{(1)} \rangle$. Therefore the differential of the sclaing automorphism $\hat s_t$ is the diagonal matrix
%
%
with $(\dim V_1 - j)$ many $t$'s on the diagonal and $\dim V_i$ many $t^i$'s on the diagonal for $2 \le i \le d$. So the determinant of this matrix is $t^{D-j}$ and this implies the scaling behaviour.\\
Then the boundary $\partial b$ of $b$ consists of all unit cubes contained in $b$ of dimension one smaller than $b$ and with one additional coordinate set $1$ or $0$. This further by now constant coordinate has scaled under the scaling automorphism at least linearly, and so we get 
$$\mass(s_t(\partial b)) \le t^{D-j-1} \cdot \mass(\partial b)\ .$$
Next we need to construct the $G$-invariant, closed $m$-form $\gamma$. We do this by choosing $\gamma$ as the volume form  of $b$:
$$\gamma=(v_{j}^{(1)})^*\wedge ... \wedge (v_{\dim V_d}^{(d)})^* \ .$$
Here $v^*$ denotes the dual form of $v \in \mf g$. This form $\gamma$ is by definition $G$-invariant and 
$$\int_b \gamma = \mass(b) >0$$
as $\gamma$ is the volume form of $b$.\\
For the scaling behaviour of $\gamma$ we get
$$s_t^*\gamma=t^{D-j}\gamma $$
with the same argument as above for the scaling behaviour of the cube $b$.\\\quad\\
So it remains to show, that $\gamma$ is closed. For that, we recall the formula for the differential of a $G$-invariant $p$-form $\omega$:
$$(p+1)!(\opna d \omega)(X_0,...,X_p)=\sum_{s<t}(-1)^{s+t+1}\omega([X_s,X_t],X_0,...,\widehat X_s,...,\widehat X_t,...,X_p)\ .$$ 
The key essence of this formula is that it suffices to examine ``pre-images" of basis vectors under the Lie bracket to compute the differential of $\gamma$. This follows as a dual form $v^*$ only sees the projection to the subspace spanned by $v$. To be precise, the differential of $v^*$ is  
$$2\cdot(\opna d\! v^*)(X,Y)\stackrel{(1)}=v^*([X,Y])\stackrel{(2)}= \sum_{u,w \in B \atop [u,w]=v} u^*(X)\cdot w^*(Y)= \sum_{u,w \in B \atop [u,w]=v} (u^*\wedge w^*)(X,Y)$$
where $(1)$ holds by the above formula and $(2)$ is the evaluation by computing the ``length'' of the projection to the subspace $\langle v \rangle \subset \mf g$.\\
For the left-invariant $m$-form $\gamma$ this leads to
$$ \opna d \! \gamma= \frac{1}{(m\!+\!1)!}\sum_{i=1}^d\sum_{q=1}^{\dim V_i}\!\!\sum_{v,w \in B \atop [v,w]=v_q^{(i)}} \!\!(\text{-}1)^{i+q+1}v^*\wedge w^* \wedge (v_{j+1}^{(1)})^* \wedge ... \wedge \widehat{(v_{q}^{(i)})^*} \wedge ... \wedge (v_{\dim V_d}^{(d)})^*.$$
Now let $i\ge 2$ (the case $i=1$ is trivial, as $V_1$ has zero intersection with the image of the Lie bracket). Let $v_q^{(i)} \in V_j$ be one of the above chosen basis vectors. For each pair 
$$x=\sum_{s=1}^d\sum_{t=1}^{\dim(V_s)} \alpha_{st}v_t^{(s)} \ ,\quad y=\sum_{s=1}^d\sum_{t=1}^{\dim(V_s)} \beta_{st}v_t^{(s)} \in \mf g$$
with $[x,y]=v_q^{(i)}$ we get:
$$[x,y]=\sum_{s_1,s_2=1}^d\sum_{t_1,t_2=1}^{\dim(V_s)} \alpha_{s_1t_1}\beta_{s_2t_2}[v_{t_1}^{(s_1)},v_{t_2}^{(s_2)}]\ .$$
By this and by the linearity of differential forms we can assume without loss of generality, that $x$ and $y$ are basis vectors in $B$.\\
We first look at the case $i=2$. In this case we have $x,y \in B_1$ as $[V_1,V_u] = V_{u+1}$. The first $k+1$ vectors $v_1^{(1)},...,v_{k+1}^{(1)}$ span the $\Omega$-isotropic subspace $S\subset V_1$ and therefore $[v_s^{(1)},v_t^{(1)}]=0$ for $1\le s,t \le k+1$. So at least one of the vectors $x$ and $y$ has to be in $B_1 \setminus \{v_1^{(1)},...,v_{k+1}^{(1)}\}$. But all dual forms of basis vectors $v_\ell^{(1)}$ with $\ell > k+1 >j-1$ are part of $\gamma$ (and are not the deleted $v_i^{(2)}$ ). Therefore all these summands in $\opna d \! \gamma$ are zero.\\
Now let $i\ge3$. Then at least one of the vectors $x$ and $y$ has to be a basis vector in a layer $V_\ell$ with $2 \le \ell \le i-1$. But again all the dual forms of such basis vectors are part of $\gamma$ and unequal to the deleted $v_q^{(i)}$. Therefore all these summands in  $\opna d \! \gamma$ are zero, too.\\
This implies 
$$\opna d \! \gamma = 0$$
and we can apply Burillo's theorem and gain
$$F^{n-j}_G(l) \succcurlyeq l^\frac{D-j}{D -j-1}$$
which is the desired bound.
\end{proof}

%
%

\subsubsection{The proof of Theorem \ref{Thm3}}
\vspace*{-5mm}
Theorem \ref{Thm3} reduces the conditions of Theorem \ref{Thm1} and Theorem \ref{Thm2} in the case of $2$-step nilpotent Lie groups. It omits the condition of the existence of the lattice $\Gamma$, but doesn't weaken the statements about the filling functions.
Therefore Theorem \ref{Thm3} follows from Theorem \ref{Thm1} and Theorem \ref{Thm2} if in every $2$-step nilpotent Lie group there exists a lattice with the requested scaling property. So we only have to prove the following lemma:

\begin{lem}\label{lattice}
Let $G$ be a simply connected $2$-step nilpotent Lie group with Lie algebra $\mf g$. Then there exists a lattice $\Gamma \subset G$ with $s_2(\Gamma)\subset \Gamma$. 
\end{lem}

\begin{proof}
Let $\mf g = V_1 \oplus V_2$ be the grading of the Lie algebra. Then take a basis 
$$B_1=\{v_1^{(1)},v_1^{(2)},...,v_1^{(\dim V_1)}\}$$
of the first layer. For $V_2$ we complete $\{\frac{1}{2}[a,b] \mid a,b \in B_1\}$ to a basis 
$$B_2=\{v_2^{(1)},v_2^{(2)},...,v_2^{(\dim V_2)}\}$$
of $V_2$. This leads to a Basis $B=B_1 \cup B_2$ of the Lie algebra $\mf g$.\\
Now let 
$$\mc Z\defeq  \langle \{b \mid b \in B\} \rangle_\ZZ \subset \mf g$$
be the $\ZZ$-span of this basis. Then $\mc Z$ is, by construction, closed under the Lie bracket $[\cdot, \cdot]$ and fulfils $L(s_2)(\mc Z) \subset \mc Z$ as $2t\in \ZZ$ for all $t \in \ZZ$.\\
Then define 
$$ \Gamma \defeq  \langle \exp(\mc Z) \rangle\ .$$
Then $\Gamma \le G$ is a lattice :\\
The structural constants with respect to the basis $B$ of $\mf g$ are rational. The set $\mc Z$ is lattice of maximal rank in the $\QQ$-span $\mf g_\QQ$ of $B$. Therefore the group generated by $\exp(\mc Z)$ is a lattice in $G$ (see \cite[Theorem 2.12]{Raghunathan}).\\\quad\\
The Baker-Campbell-Hausdorff formula in the case of $2$-step nilpotent Lie groups reduces to 
$$\exp(x)\exp(y)=\exp(x+y+\frac{1}{2}[x,y])$$
and so for $h=\exp(x),g=\exp(y) \in \Gamma$ with $x,y \in \mc Z$ holds:
$$hg=\exp(x)\exp(y)=\exp(x+y+\frac{1}{2}[x,y])$$
The product $hg$ is in $\exp(\mc Z)$ as $\frac{1}{2}[x,y] \in \mc Z$. (One can see this by writing $x=\sum_{b\in B} \alpha_b b$ and $y=\sum_{b\in B} \beta_b b$). Therefore holds $ \Gamma = \exp(\mc Z) $.\\\quad\\
Every $g\in \Gamma$ can be written as 
$$g=\exp( \sum_{b\in B_1} n_b b + \sum_{b\in B_2} m_b b)$$
and so 
$$s_2(g)=\exp\big(L(s_2)(\sum_{b\in B_1} n_b b + \sum_{b\in B_2} m_b b)\big)=\exp(\sum_{b\in B_1} 2n_b b + \sum_{b\in B_2} 4m_b b) \in \Gamma\ .$$
Therefore $\Gamma$ fulfils $s_2(\Gamma) \subset \Gamma$.
\end{proof}

\subsubsection{The proof of Theorem \ref{Thm4}}
\vspace*{-5mm}

Recall that we denote the set of integral $m$-currents by $\boldsymbol I_m(G)$ and set of integral $m$-currents with compact support by $\boldsymbol I_m^{cpt}(G)$.\\\quad\\
For the proof of Theorem \ref{Thm4} we need two propositions. The first is due to Wenger:

\begin{prop}[{\cite[Proposition 3.6]{Wenger11}}]\label{II}
Let $G=(G,\opna d)$ be a stratified nilpotent Lie group equipped with a left-invariant Riemannian metric and denote by $G_\infty=(G,\opna d_c)$ the same group equipped with its Carnot-Carath\'eodory metric. If $G$ satisfies an Euclidean isoperimetric inequality for $\boldsymbol I_m^{cpt}(G)$, then $G_\infty$ satisfies an Euclidean isoperimetric inequality for $\boldsymbol I^{cpt}_m(G_\infty)$.
\end{prop}

The original proposition in \cite{Wenger11} presumes an Euclidean isoperimetric inequality for $\boldsymbol I_m(G)$. But an examination of the proof yields, that only an Euclidean isoperimetric inequality for $\boldsymbol I_m^{cpt}(G)$ is needed.\\
\quad\\
Further we need the following proposition:
\begin{prop}\label{FI}
Let $G$ be a Lie group equipped with a left-invariant Riemannian metric. Let $m\in \NN$ and $\delta \ge 1$. 
If $F_G^{m+1} \preccurlyeq l^\delta$, then $G$ satisfies an isoperimetric inequality of rank $\delta$ for $\boldsymbol I_m^{cpt}(G)$.
\end{prop}
\begin{proof}
Let $T \in \boldsymbol I_m^{cpt}(G)$ with $T \ne 0$ and $\partial T=0$. Note, that in particular $\partial T$ is associated to a Lipschitz chain, i.e. $\partial T= a_\#$ for the Lipschitz chain $a=0$. Now embed $G$ isometrically in some $\RR^N$  and look at $T$ and $\partial T$ from now on as integral currents of $\RR^N$.\\
Let $\eta >0$ be arbitrary small. With \cite[Lemma 5.7]{FF60} we get:\\
There is an integral current $ S \in \boldsymbol I_{m+1}(\RR^N)$, such that 					\vspace*{-3mm}
\begin{enumerate}[i)]
\item $T-\partial S $ is a Lipschitz chain,											\vspace*{-3mm}
\item $\boldsymbol N(S) = \boldsymbol M(S)+\boldsymbol M(\partial S)\le \eta$ \quad and		\vspace*{-3mm}
\item $\opna{spt}(S) \subset U_\eta(\opna{spt}(T))$.									\vspace*{-3mm}
\end{enumerate}
As $G$ is an isometrically embedded Riemannian manifold, it is a local Lipschitz neighbourhood retract. This means, there is a neighbourhood $U$ of $G$ in $\RR^N$ and a locally Lipschitz map
$$\varphi: U \to G \ \text{ with } \ \varphi(g)=g \quad \forall g \in G.$$
Now let $\eta >0$ be sufficiently small, such that $\opna{spt}(S) \subset U$. Then the map $\bar \varphi \defeq \varphi_{|\opna{spt}(S)}:\opna{spt}(S)\to G$ \ is locally Lipschitz.\\\quad\\
For $x\in \opna{spt}(S)$ we denote by $U(x)$ the maximal neighbourhood of $x$ such that $\bar \varphi$ is Lipschitz on $U(x)$.\\
As $T-\partial S$ is associated a Lipschitz chain, it has compact support. Therefore there is a finite cover 
$$ \bigcup_{i=1}^d U(x_i) \quad \text{ with } \quad x_1,...,x_d \in \opna{spt}(T -\partial S)\ .$$
So one can subdivide $T-\partial S$ in finitely many smaller simplices (i.e. simplicial chains $\alpha_j: \Delta \to \opna{spt}(T-\partial S)$), such that each $\alpha_j(\Delta)$ is completely contained in one of the $U(X_i)$. \\
As now $\bar \varphi$ is Lipschitz on each $\alpha_j(\Delta)$, we obtain the Lipschitz chain
$$\varphi^*(T-\partial S)=\sum_{j=1}^l \bar \varphi\circ \alpha_j$$
on $G$.\\\quad\\
Further the support of $T$ is compact. Therefore the closure of the $2 \eta$-neighbourhood of $\opna{spt}(T)$ is compact. So we can divide $\opna{spt}(S)$ in finitely many Borel sets $A_v$, such that each of these sets is completely contained in one of the neighbourhoods $U(x) ,\,  x \in \opna{spt}(S)$. One can do this as $\RR^N$ is second-countable. Then, as $\bar \varphi$ is Lipschitz on each $A_v\subset U(x_v)$, we obtain the integral current
$$\varphi^*(S)= \sum_{v=1}^n \bar\varphi_\#(S\llcorner A_v)$$
on $G$.\\\quad\\
Further holds
$$\mass(\varphi^*(T-\partial S)) \sim \boldsymbol M(T)$$
as $\bar \varphi$ is Lipschitz on each simplex and $\boldsymbol M(\partial S) \le \eta$.\\
\quad\\
Now let $l:= \boldsymbol M(T)$ and let $b$ be a Lipschitz chain in $G$ with $\partial b=\varphi^*(T-\partial S)$ and $\mass(b) \preccurlyeq l^\delta$. Such a chain exists due to the condition on the filling function. So we get 
\begin{align*}
\partial (b_\# + \varphi^*(S))&=\partial b_\# + \partial \varphi^*(S)=\partial b_\# +  \varphi^*(\partial S)\\&= \varphi^*(T-\partial S)+\varphi^*(\partial S)=T-\varphi^*(\partial S)+\varphi^*(\partial S)\\&=T
\end{align*}
and
$$\boldsymbol M(b_\#+\varphi^*(S)) \preccurlyeq l^\delta$$
as $\bar \varphi$ is Lipschitz on each restriction and $\boldsymbol M(S) \le \eta$. 
\end{proof}

Now we are prepared for the proof of Theorem \ref{Thm4}. We start with the case, that condition \textit{a)} is fulfilled.
\begin{prop}\label{propa)}
Let $G$ be a stratified nilpotent Lie group equipped with a left-invariant Riemannian metric. Further let $\mf g$ be the Lie algebra of $G$  with grading $\mf g=V_1\oplus ... \oplus V_d$. Let $k_0,k_1\in \NN$, such that $(k_0+1)$ is the maximal dimension of an $\Omega$-regular, $\Omega$-isotropic subspace of $V_1$ and $(k_1+1)$ is the maximal dimension of an  $\Omega$-isotropic subspace of $V_1$. If there is an $k_0 \le k \le k_1$, such that there is an integral current $T\in \boldsymbol I^{cpt}_{k+1}(G,\opna d_c)$ with $\partial T=0$ and $T\ne 0$, but no integral current $S \in \boldsymbol I^{cpt}_{k+2}(G,\opna d_c)$ with $\partial S=T$, then holds
$$F^{k+2}_G(l) \succnsim l^\frac{k+2}{k+1} \ .$$
\end{prop}

\begin{proof}
So there is an $k_0 \le k \le k_1$ such that there is an integral current $T\in \boldsymbol I^{cpt}_{k+1}(G,\opna d_c)$ with $\partial T=0$ and $T\ne 0$ but no integral current $S \in \boldsymbol I^{cpt}_{k+2}(G,\opna d_c)$ with $\partial S=T$. This means, that $(G,\opna d_c)$ doesn't satisfy an isoperimetric inequality of rank $\delta$ for $\boldsymbol I^{cpt}_{k+1}(G,\opna d_c)$ for any $\delta < \infty$, and in particular no Euclidean isoperimetric inequality for $\boldsymbol I^{cpt}_{k+1}(G,\opna d_c)$. By Proposition \ref{II}, $G$ doesn't satisfy an Euclidean isoperimetric inequality  for $\boldsymbol I_{k+1}(G)$ and by Proposition \ref{FI} we have 
$$F^{k+2}_G(l) \succnsim l^\frac{k+2}{k+1}$$
as desired.
\end{proof}

We finish the proof of Theorem \ref{Thm4} by showing, that if condition \textit{b)} is fulfilled this implies that condition \textit{a)} is satisfied.

\begin{lem}\label{lemb)}
Let $G$ be a stratified nilpotent Lie group equipped with a left-invariant Riemannian metric. Further let $\mf g$ be the Lie algebra of $G$  with grading $\mf g=V_1\oplus ... \oplus V_d$. Let $k_0,k_1\in \NN$, such that $(k_0+1)$ is the maximal dimension of an $\Omega$-regular, $\Omega$-isotropic subspace of $V_1$ and $(k_1+1)$ is the maximal dimension of an  $\Omega$-isotropic subspace of $V_1$. 
If the two numbers $k_0$ and $k_1$ coincide, then there is for $k \defeq k_0=k_1$ an integral current $T\in \boldsymbol I^{cpt}_{k+1}(G,\opna d_c)$ with $\partial T=0$ and $T\ne 0$ but no integral current $S \in \boldsymbol I^{cpt}_{k+2}(G,\opna d_c)$ with $\partial S=T$.
\end{lem}

\begin{proof}
Suppose $k+1$ is the maximal dimension of $\Omega$-regular, $\Omega$-isotropic subspaces of $V_1$. Using the $h$-principle (as in the proof of Theorem \ref{Thm1}), we can construct an $(k+1)$-horizontal triangulation of $G$. The boundary of any $(k+2)$-simplex $\Delta^{(k+2)}$ forms an horizontal Lipschitz $(k+1)$-chain $a=\partial \Delta^{(k+2)} \ne 0$ with $\partial a=0$. Viewed as current, this gives us the integral current $T:= a_\# \in \boldsymbol I^{cpt}_{k+1}(G,\opna d_c)$. 
As $k+1$ is the maximal dimension of  $\Omega$-isotropic subspaces of $V_1$ too, we have by \cite[Theorem 1.1]{Magnani}, that $(G,\opna d_c)$ is purely $\mc H^{k+2}$-unrectifiable and so there are no non-trivial integral $(k+2)$-dimensional currents on $(G,\opna d_c)$. So condition \textit{a)} holds.
\end{proof}

The combination of Proposition \ref{propa)} and Lemma \ref{lemb)} proves Theorem \ref{Thm4}.

%
%

\subsection{Geometrical Interpretation}
\vspace*{-5mm}
All of our theorems have conditions concerning the existence of $\Omega$-regular, $\Omega$-isotropic subspaces in the first layer of the grading of Lie algebra. The proofs used this algebraic conditions in a more or less technical manner. But it is interesting what the geometrical meaning of these subspaces is. Furthermore, one can explain geometrically the change of the behaviour of the filling invariants at the maximal dimension of such subspaces.

A good grasp of the meaning of the maximal dimension of an $\Omega$-regular, $\Omega$-isotropic subspace $S\subset V_1$ one can get from the view point of differential geometry. More explicitly, one has to look at the sectional curvature. For the sectional curvature of a Lie group equipped with a left-invariant Riemannian metric, there is the following formula (see \cite{Milnor}):

Let $G$ be a Lie group with Lie algebra $\mf g$ and let $\{e_1,...,e_n\}$ be an orthonormal basis of $\mf g$. Define the numbers $\alpha_{uvw}$ by
$$[e_u,e_v]= \sum_{w=1}^n \alpha_{uvw} e_w\ .$$
Then holds for the sectional curvature $K$:
\begin{align*}
K(e_i,e_j)= \sum_{k=1}^n \Big(\frac{1}{2} &\alpha_{ijk}(- \alpha_{ijk}+\alpha_{jki}+\alpha_{kij})\\
&- \frac{1}{4}(\alpha_{ijk}-\alpha_{jki}+\alpha_{kij})(\alpha_{ijk}+\alpha_{jki}-\alpha_{kij})-\alpha_{kii}\alpha_{kjj}\Big)
\end{align*}

Now consider $G$ as $2$-step nilpotent with the grading
$$\mf g = V_1 \oplus [V_1,V_1]$$
of its Lie algebra. Further let $\{e_1,...,e_{n_1}\}$ be an orthonormal basis of $V_1$ and $\{e_{n_1+1},...,e_{n}\}$ an orthonormal basis of $V_2:=[V_1,V_1]$. As $G$ is $2$-step nilpotent, one gets
$$[e_u,e_v]=0\quad \text{if \ $u \ge n_1+ 1$ \ or \ $v\ge n_1+ 1$}$$
and $[e_u,e_v]\in V_2$ for all $v,u \in \{1,...,n\}$. Therefore $\alpha_{uvw}=0$, whenever $u \notin \{1,...,n_1\}$ or $v \notin \{1,...,n_1\}$ or $w\le n_1$. 

For the sectional curvature in the case $i,j \le n_1$ follows:
\begin{align*}
K(e_i,e_j)&= \sum_{k=n_1+1}^n \Big( \frac{1}{2} \alpha_{ijk}(-\alpha_{ijk}+0+0)-\frac{1}{4}(\alpha_{ijk}-0+0)(\alpha_{ijk}+0-0)-0 \cdot 0 \Big)\\
&=\sum_{k=n_1+1}^n \Big( -\frac{1}{2}(\alpha_{ijk})^2-\frac{1}{4}(\alpha_{ijk})^2 \Big)\\
&=-\frac{3}{4}\sum_{k=n_1+1}^n (\alpha_{ijk})^2 \\
&=:K_{1,1}
\end{align*}
And in the case $i\le n_1$ and $j \ge n_1+1$:
\begin{align*}
K(e_i,e_j)&= \sum_{k=1}^{n_1} \Big( 0(-0+0+\alpha_{kij})-\frac{1}{4}(0-0+\alpha_{kij})(0+0-\alpha_{kij})-0 \cdot 0 \Big)\\
&=\sum_{k=1}^{n_1} \frac{1}{4}(\alpha_{kij})^2 \\
&=\frac{1}{4}\sum_{k=1}^{n_1} (\alpha_{kij})^2\\
&=:K_{1,2}
\end{align*}
In the case, that both vectors are from the basis of $V_2$, the sectional curvature equals $0$.\\
\quad\\
One can see, that $K_{1,1}\le 0$ with equality if and only if $[e_i,e_j]=0$; and $K_{1,2} \ge 0$ with equality if and only if for all $k\in\{1,...,n_1\}$ holds: $\pi_{e_j}([e_k,e_i])=0$, where $\pi_{e_j}$ denotes the projection on the subspace $\langle e_j \rangle$.

Let $S\subset V_1$ be an $\Omega$-isotropic, $\Omega$-regular subspace of maximal dimension, say of dimension $m$, and let the basis $\{e_1,...,e_n\}$ be chosen in way, such that $S=\langle e_1,...,e_m \rangle$. Then one gets:

\begin{enumerate}[1)]
\item $K(e_i,e_j)=0$ if $i,j \le m.$

\item For all $e_j$ with $m+1 \le j \le n_1$ there exists an $i \in \{1,...,m\}$, such that $K(e_i,e_j) < 0$.

\item For all $e_j$ with $j \ge n_1+1$ there exists an $i \in \{1,...,m\}$, such that $K(e_i,e_j) > 0$.
\end{enumerate}
The first property comes by the $\Omega$-isotropy, the second by the maximality of the dimension and the third by the $\Omega$-regularity.

This shows, that every plane in $S$ has sectional curvature $=0$. But whenever one extends $S$ by another direction, one gets a plane with sectional curvature $\ne 0$.

So one can explain the Euclidean behaviour of the filling invariants up to the maximal dimension of $\Omega$-regular, $\Omega$-isotropic subspaces by the flatness of these subspaces. The super-Euclidean behaviour in the dimension above is related to the positive curvature, which occurs whenever one adds a direction not contained in the first layer of the grading. 

Another way to see the necessity of the $\Omega$-regularity is the following example:

For $n\ge 4$ the group  $N_n$ of unipotent upper triangular $(n\times n)$-matrices the $2$-dimensional filling function fulfils 
$$F_{N_n}^2(l) \succcurlyeq l^3 \nsim l^\frac{1+1}{1}$$
which is a strictly super-Euclidean behaviour (see  \cite{Burillo}). 
But this is no contradiction to \mbox{Theorem \ref{Thm1}}, as the first layer of the grading $\mf n_n=V_1 \oplus ... \oplus V_{n-1}$ has dimension $n$ and the dimension of $\mf n_n$ is $\frac{n(n-1)}{2}$. \\
Therefore 
$$\dim V_1 - m = n-m \ge m(\frac{n^2-3n}{2}) =m(\frac{n(n-1)}{2}-n)=m(\dim \mf n_n - \dim V_1)$$
holds never true for $m\ge 2$. By the discussion after Theorem \ref{Thm1}, there can't exist a $2$-dimensional $\Omega$-isotropic, $\Omega$-regular subspace $S$ of $V_1$ and therefore  \mbox{Theorem \ref{Thm1}} doesn't apply.\\
On the other hand, there is a $\lfloor \frac{n}{2} \rfloor$-dimensional $\Omega$-isotropic subspace of $V_1$, generated by the matrices $E_{2k-1,2k}=(e_{i,j})$, $1\le k \le \lfloor \frac{n}{2} \rfloor$, with only non-zero entry $e_{2k-1,2k}=1$.\\
This shows, as $\lfloor \frac{n}{2} \rfloor \ge 2$ for $n\ge 4$, that the condition of the $\Omega$-regularity is of crucial importance for the Euclidean behaviour of the filling functions.

%
%

\section{Higher divergence functions of stratified nilpotent Lie groups}\label{SectionHDF}
\vspace*{-5mm}
Our results for the filling functions of stratified nilpotent Lie groups lead directly to lower bounds for the higher divergence functions. This is, as we will see in the proof, mainly due to the left-invariance of the Riemannian metric. 

\vspace*{-5mm}
\subsection{Results}
\vspace*{-5mm}
In the low dimensions we obtain the following theorem:

\begin{Thm}[{Divergence functions in low dimensions}]\label{Thm5}
Let $G$ be a stratified nilpotent Lie group equipped with a left-invariant Riemannian metric. Further let $\mf g$ be the Lie algebra of $G$  and $V_1$ be the first layer of the grading and let $k \in \mathbb N$. If there exists a  lattice $\Gamma\subset G$ with $s_2(\Gamma)\subset\Gamma$ and  a $(k+1)$-dimensional $\Omega$-isotropic, $\Omega$-regular subspace $S\subset V_1$, then holds:
$$\opna{Div}_G^j(r) \succcurlyeq r^{j+1} \quad \text{for all } j\le k.$$
\end{Thm}

In the high dimensions we obtain additional upper bounds for the higher divergence functions, which coincide with the lower bounds. The lower bounds come in the same way as in the low dimensions, while the upper bounds are possible to establish because the filling functions are sub-Euclidean (in contrast to the Euclidean filling functions in the low dimensions). \\
To establish the upper bounds on the higher divergence functions it is important to know the divergence dimension of the stratified nilpotent Lie group. Every simply connected nilpotent Lie group of dimension $n$ is polynomial Lipschitz equivalent to $\RR^n$ via the exponential map $\exp: \RR^n \cong \mf g \to G$. This means that the Lipschitz constants of $\exp$ and $\exp^{-1}$ on balls of radius $R$ grow at most polynomial in $R$. Using this, one can construct in $G$ ($\rho r$-avoidant) fillings with polynomial bounded mass of ($r$-avoidant) cycles from Euclidean ones (compare \cite[Section 5]{GGT}) and vice versa. Therefore $G$ and $\RR^n$ have the same divergence dimension: $\opna{divdim}(G)=n-2$.

\begin{Thm}[{Divergence functions in high dimensions}]\label{Thm6}
Let $G$ be an $n$-dimensional stratified nilpotent Lie group equipped with a left-invariant Riemannian metric. Further let $\mf g$ be the Lie algebra of $G$  with grading $\mf g=V_1\oplus ... \oplus V_d$. Denote by $D =  \displaystyle{\sum_{i=1}^d i \cdot \dim V_i}$ \  the Hausdorff-dimension of the asymptotic cone of $G$ and let $k\in \NN$. If there exists a  lattice $\Gamma\subset G$ with $s_2(\Gamma)\subset\Gamma$ and  a $(k+1)$-dimensional $\Omega$-regular, $\Omega$-isotropic subspace $S \subset V_1$, then holds:
$$\opna{Div}^{n-j}_G(r) \sim r^\frac{(D-j)(n-j-1)}{D -j-1} \quad \text{for all } 2 \le j \le k.$$
\end{Thm}

As we can reduce the conditions of Theorem \ref{Thm1} and Theorem \ref{Thm2} in the case of simply connected $2$-step nilpotent Lie groups (see Theorem \ref{Thm3}),
we can do the same for Theorem \ref{Thm5} and Theorem \ref{Thm6} and obtain:

\begin{Thm}[{Divergence functions of $2$-step nilpotent Lie groups}]\label{Thm7}
Let $G$ be an $n$-dimensional simply connected $2$-step nilpotent Lie group equipped with a left-invariant Riemannian metric. Further let $\mf g$ be the Lie algebra of $G$  with the grading $\mf g = V_1 \oplus V_2$, let $n_2= \dim V_2$ and let $k \in \mathbb N$. If there exists a   $(k+1)$-dimensional $\Omega$-regular, $\Omega$-isotropic subspace $S\subset V_1$,  then holds:
\begin{enumerate}[(i)]
\item $\opna{Div}_G^j(r) \succcurlyeq r^{j+1}\hspace{2.9cm} $ for all $j \le k$,

\item $\opna{Div}^{n-j-1}_G(r) \sim r^\frac{(n+n_2-j)(n-j-1)}{n+n_2 -j-1} \quad$  for all $2\le j \le k$.

\end{enumerate}
\end{Thm}

\vspace{2mm}

%
%

\subsection{Proofs for Higher Divergence Functions}\label{S6}
\vspace*{-5mm}
We proceed with the statements about the higher divergence functions. 

The statements about the higher divergence functions of a stratified nilpotent Lie group equipped with a left-invariant Riemannian metric  can be naturally divided into two parts: The lower bounds and the upper bounds.
Due to this subdivision we split the proof in the separated treatment of lower and upper bounds. To establish the lower bounds we will prove the more general Proposition \ref{Divlow}, which deduces lower bounds on the higher divergence functions from lower bounds on the filling functions in the setting of arbitrary Lie groups.
Our technique for the upper bounds on the higher divergence functions (Proposition \ref{Divup}) needs sub-Euclidean upper bounds on the filling functions and so only works in the top dimensions.

\subsubsection{Lower bounds}\label{Sectionlb}
\vspace*{-5mm}
We consider a Lie group $G$ equipped with a left-invariant Riemannian metric. So if we transport a cycle by left-multiplication, there is no change of the mass of the cycle or of the mass bounded by the cycle. We use this fact to prove the following proposition, which provides the lower bounds for the higher divergence functions stated in the Theorems \ref{Thm5}, \ref{Thm6} and \ref{Thm7}.

\begin{prop} \label{Divlow}
Let $G$ be a Lie group equipped with a left-invariant Riemannian metric and let $m \in \NN$, such that $\opna{divdim}(G) \ge m$. If the filling function $F^{m+1}_G$ is bounded from below by a function $h: \RR^+ \to \RR^+$, then holds:
$$\opna{Div}^m_G(r) \succcurlyeq  h(r^m) \ .$$

\end{prop}

\begin{proof}
By assumption holds for the $(m+1)$-dimensional filling function
$$F^{m+1}_G(l) \succcurlyeq h(l) \ .$$
So there is a constant $C \ge 1$ such that for every $l >0$ there exists a $m$-cycle $a_l$ with mass $\mass(a_l) = l$  such that every $(m+1)$-chain $b$ with boundary $\partial b =a$ has to fulfil 
$$C \cdot \mass(b) + Cl +C \ge h(l) \ .$$
Let $\alpha_0=1$ and $\rho_0<1$, such that $G$ is $(\rho_0,m)$-acyclic at infinity. We have to show, that there are constants $L,M \ge 1$, such that for all $\rho \le \rho_0$ and all $\alpha \ge \alpha_0$ there is a constant $A \ge 1$ with:
$$A \cdot \opna{div}^m_{L\rho, M\alpha}(Ar + A) + O(r^m) \ge h(r^m) \ .$$
To do this, we use the above 'hard-to-fill'-cycle $a_l$ with $l= r^m$.\\ 
As $a_l$ is a priori not $r$-avoidant, we transport it out of the $r$-ball around $1 \in G$ by left-multiplication with some suitable group element $g\in G$ and obtain the $r$-avoidant $m$-cycle $g\bullet a_l$. For the mass we obtain
$$\mass(g \bullet a_l)= \mass(a_l)=l$$
as the metric is left-invariant. \\
The left-invariance of the metric also guarantees that the property 'hard-to-fill' is preserved under the left-multiplication, i.e. every $(m+1)$-chain $b$ with boundary $\partial b= g \bullet a_l$ has 
$$\mass(b) \succcurlyeq h(l)\ .$$
As the $\rho r$-avoidance of the filling is an additional restriction to the $(m+1)$-chain, the above inequality for the mass of $b$ holds true for $\rho r$-avoidant $(m+1)$-chains $b$ with boundary $\partial b=g \bullet a_l$. Here we need the assumption $\opna{divdim}G\ge m$ for the existence of such $\rho r$-avoidant fillings.\\
Now we choose $L=M=1$ and for $\rho \le \rho_0$ and $\alpha \ge \alpha_0$ we set $A=\alpha \cdot C$. \\
Then we get
\begin{align*} 
A \cdot \opna{div}^m_{L\rho, M\alpha}(Ar + A) + O(r^m)  &\ge A \cdot \opna{div}^m_{\rho, \alpha}(Ar +A) + A \cdot r^m + A \\
& \stackrel{\mathrm{(1)}}\ge C \cdot \opna{div}^m_{\rho, 1}( r) +  C \cdot r^m + C\\
& \stackrel{\mathrm{(2)}}\ge C \cdot \mass(b_0) + C \cdot r^m + C\\
&\stackrel{\mathrm{(3)}}\ge h(r^m)
\end{align*}
where $(1)$ holds true, as $\alpha \ge 1$, $A\ge C\ge 1$ and as $\opna{div}^m_{\rho, \alpha}(r)$ is increasing in $\alpha$ and $r$. Further $(2)$ holds true, as we have the $r$-avoidant 'hard-to-fill' $m$-cycle $a_l$ with $\mass(a_l)= l = r^m \le \alpha \cdot r^m$ and as $\opna{divdim}(G)\ge m$ there is some optimal $\rho r$-avoidant filling $b_0$ of $a_l$. And $(3)$ holds true, as $b_0$ is a filling of $a_l$ and therefore has to fulfil this inequality by the lower bound on the filling function.
\end{proof}

With this Proposition \ref{Divlow} and our results for the filling functions (Theorem \ref{Thm1}, Theorem \ref{Thm2} and Theorem \ref{Thm3}) we get the lower bounds in Theorem \ref{Thm5}, Theorem \ref{Thm6} and Theorem \ref{Thm7}.

%
%

\subsubsection{Upper bounds}
\vspace*{-5mm}

To obtain the upper bounds for the higher divergence functions, we prove a more general proposition: We deduce sub-Euclidean upper bounds for higher divergence functions from sub-Euclidean upper bounds for filling functions.

\begin{prop} \label{Divup}
Let $M$ be a complete Riemannian manifold and $m \in \mathbb N$, $m \le \opna{divdim}(M)$. If there is a $\delta < \frac{m+1}{m}$, such that the $(m+1)$-dimensional filling function  $F_M^{m+1}(l)$ is bounded from above by $\ l^\delta$, then holds:
$$\opna{Div}^m_M(r) \preccurlyeq r^{\delta  m}\ .$$
\end{prop}

\begin{proof}
First choose a basepoint $x_0 \in M$ and let $\alpha \ge 1$ and $\rho_0=\frac{1}{4}$. Let $r_0>0$ be sufficiently large. Further let $C>0$, such that 
$$F_M^{m+1}(\alpha r^m) \le C \cdot (\alpha r^m)^\delta \qquad \forall r \ge r_0.$$
Let $a$ be an $r$-avoidant Lipschitz $m$-cycle of $\mass(a)\le \alpha r^m \eqdef l$.\\
The condition on the $(m+1)$-dimensional filling function of $M$ implies the existence of a Lipschitz $(m+1)$-chain $b$ (not necessarily avoidant) with \vspace{-3mm}
\begin{enumerate}[i)]
\item \quad $\partial b=a $ \quad and 			\vspace{-3mm}
\item \quad $\mass(b) \le C \cdot l^\delta\ .$	\vspace{-3mm}
\end{enumerate}
Now let $T=a_\#$ be the Lipschitz cycle $a$ considered as integral current. Then there is an integral current $S \in \boldsymbol I_{m+1}(M)$ with \vspace{-3mm}
\begin{enumerate}[i)]
\item \quad $\partial S=T$\quad and 					\vspace{-3mm}
\item \quad $\boldsymbol M(S) \le C \cdot l^\delta\ .$ 		\vspace{-3mm}
\end{enumerate}
For example, $S=b_\#$ would be an appropriate choice.\\
So the proof of  \cite[Lemma 3.1]{Wenger06} (see also \cite[Lemma 3.4]{Wenger05}) together with the computation in the proof of \cite[Proposition 1.8]{Wenger06} yields the following:

For $x \in M$ and $t>0$ denote by $B(x,t)$ the closed ball of radius $t$ around $x$.\\ 
For sufficiently large $r$ and every $\varepsilon >0$ there is an integral current $S_\varepsilon \in  \boldsymbol I_{m+1}(M)$ with: \vspace{-3mm}
\begin{enumerate}[i)]
\item \quad $\partial S_\varepsilon=T $ , 								\vspace{-3mm}
\item \quad $\boldsymbol M(S_\varepsilon) \le C \cdot l^\delta + \varepsilon$ , 	\vspace{-3mm}
\item \quad $\opna{spt} S_\varepsilon \subset M \setminus B(x_0, \frac{1}{2}r)$ . \vspace{-3mm}
\end{enumerate}

Now we have an $\frac{1}{2}r$-avoidant integral current $S_\varepsilon$, that ``fills" the Lipschitz cycle $a$. From this integral current we construct a $\rho_0r$-avoidant Lipschitz $(m+1)$-chain of the desired mass that fills $a$ as follows: \\
\quad\\
For $A \subset \RR^N$ and $t>0$ denote by $B(A,t) \defeq \{y \in \RR^N \mid \exists x \in A: \opna \|x-y\|_2 \le t\}$.\\
At first we use the Nash embedding theorem to embed $M$ isometrically in some $\RR^N$. Consider $S_\varepsilon$ as an integral current in $\RR^N$.  Then \cite[Lemma 5.7]{FF60} provides for every $\eta >0$ the existence of a Lipschitz $(m+1)$-chain $b'_{\varepsilon,\eta}$ such that: \vspace{-3mm}
\begin{enumerate}[i)]
\item \quad $\partial b'_{\varepsilon,\eta}=a$ ,								\vspace{-3mm}
\item \quad $\mass(b'_{\varepsilon,\eta}) \le C \cdot l^\delta + \varepsilon+ \eta$ ,	\vspace{-3mm}
\item \quad $b'_{\varepsilon,\eta} \subset  B(\opna{spt}(S_\varepsilon), \eta)$ .		\vspace{-3mm}
\end{enumerate}
As $M\subset \RR^N$ is a local Lipschitz neighbourhood retract, we can retract $b'_{\varepsilon,\eta}$  to a Lipschitz $(m+1)$-chain $b_{\varepsilon,\eta}$ on $M$ (compare proof of Proposition \ref{FI}) 
with \vspace{-3mm}
\begin{enumerate}[i)]
\item \quad $\partial b_{\varepsilon,\eta}=a$ ,													\vspace{-3mm}
\item \quad $\mass(b_{\varepsilon,\eta}) \le L^{m+1}C \cdot l^\delta + L^{m+1}\varepsilon+ L^{m+1}\eta$ ,	\vspace{-3mm}
\item \quad $b_{\varepsilon,\eta} \subset  B(\opna{spt}(S_\varepsilon), L\eta)$,							\vspace{-3mm}
\end{enumerate}
where $L$ denotes the Lipschitz constant of the retraction.\\
\quad\\
So for $\eta$ sufficiently small, i.e. such small that $ \rho_0 r <r - (L \eta + \frac{1}{2}r)= \frac{1}{2}r - L \eta$, the Lipschitz chain $b_{\varepsilon,\eta}$ is an $\rho_0r$-avoidant filling of the cycle $a$. Further holds for the mass of $b_{\varepsilon,\eta}$:
\begin{align*}
\mass(b_{\varepsilon,\eta}) &\le L^{m+1}C \cdot l^\delta + L^{m+1}\varepsilon+ L^{m+1}\eta = L^{m+1}C \cdot (\alpha r^m)^\delta + L^{m+1}\varepsilon+ L^{m+1}\eta\\
& \preccurlyeq r^{m\delta}
\end{align*}
This proves the claim.
\end{proof}

Theorem \ref{Thm2} provides for an $n$-dimensional stratified nilpotent Lie group $G$ with a lattice $\Gamma$, such that $s_2(\Gamma)\subset \Gamma$, and with an $\Omega$-isotropic, $\Omega$-regular subspace $S \in V_1$ of dimension $k+1$, the following upper bound on the filling functions:
$$F^{n-j}_G(l) \preccurlyeq l^\frac{D-j}{D -j-1} \quad \text{ for } j\le k.$$
Whenever $G$ is not abelian, i.e. not isomorphic to $\RR^n$, the Hausdorff-dimension $D$ is strictly larger than $n$. Therefore holds
$$\frac{D-j+1}{D -j} < \frac{n-j+1}{n-j}$$
and the above proposition applies for $m=n-j$ with $2\le j\le k$. Here $j$ has to be less or equal $2$, as the divergence dimension of an $n$-dimensional stratified nilpotent Lie group is $n-2$ (mentioned earlier in Section \ref{SectionHDF}).\\
\quad\\
Therefore holds 
$$\opna{Div}^{n-j}_G(r) \sim r^\frac{(D-j)(n-j-1)}{D -j-1} \quad \text{for all } 2\le  j \le k,$$
where the lower bounds are obtained Proposition \ref{Divlow} and the upper bounds are obtained by Proposition \ref{Divup}. (For the case $G\cong\RR^n$ see \cite{ABDDY}.)

Together with the lower bounds obtained in Section \ref{Sectionlb} this proves \mbox{Theorem \ref{Thm5}}, Theorem \ref{Thm6} and Theorem \ref{Thm7}.

%
%

\section{Examples}\label{S7}
\vspace*{-5mm}
To fill our results with life, we apply them to generalised Heisenberg Groups over $\CC, \HH$ and $\OO$. \\
This produces further applications to lattices in rank 1 symmetric spaces of non-compact type.

\subsection{Generalised Heisenberg Groups}\label{MHeis}
\vspace*{-5mm}
\begin{defi}
The \textup{complex Heisenberg Group} $H^n_\CC$ of dimension $2n+1$ is as manifold 
$$H^n_\CC\defeq \CC^n\times \opna{Im}\CC$$
where $\CC$ denotes the complex numbers. The group law is given by
$$(z,x)(w,y)\defeq (z+w,x+y-\frac{1}{2}\sum_{i=1}^n\opna{Im}(z_i\overline{w_i}))\ .$$
\quad\\
This is a $2$-step nilpotent Lie group with (real) Lie algebra $\mathfrak h^n_\CC=V_1\oplus V_2$ where $V_1=\CC^n,V_2=\opna{Im}\CC\cong \RR$ and with the bracket 
$$[(Z,X),(W,Y)]=(0,\sum_{i=1}^n\opna{Im}(Z_i\overline{W_i}))\ .$$
\\
It can be seen as the unique simply connected Lie group with Lie algebra generated by \\
$B\defeq \{j_1,...,j_n,k_1,...,k_n,K \}$ \\
and with the only non-trivial brackets of the generators\\
$[k,j]=K $ \quad  if both elements in the bracket have the same index.\\
\end{defi}

\begin{defi}
The \textup{quaternionic Heisenberg Group} $H^n_\HH$ of dimension $4n+3$ is as manifold 
$$H^n_\HH\defeq \HH^n\times \opna{Im}\HH$$
where $\HH$ denotes the Hamilton quaternions. The group law is given by
$$(z,x)(w,y)\defeq (z+w,x+y-\frac{1}{2}\sum_{i=1}^n\opna{Im}(z_i\overline{w_i}))\ .$$
\quad\\
This is a $2$-step nilpotent Lie group with (real) Lie algebra $\mathfrak h^n_\HH=V_1\oplus V_2$ where $V_1=\HH^n,V_2=\opna{Im}\HH$ and with the bracket %
$$[(Z,X),(W,Y)]=(0,\sum_{i=1}^n\opna{Im}(Z_i\overline{W_i}))\ .$$
\\
It can be seen as the unique simply connected Lie group with Lie algebra generated by \\
$B\defeq \{h_1,...,h_n,i_1,...,i_n,j_1,...,j_n,k_1,...,k_n,I,J,K \}$ \\
and with the only non-trivial brackets of the generators\\\quad\\
$[a,h]=A \text{ for } a\in \{i_1,...,i_n,j_1,...,j_n,k_1,...,k_n\},$ (here $A$ denotes the capital letter of the choice of $a$)\\\quad\\and\\\quad\\
$   [k,j]=I, \quad [i,k]=J,$\quad $[j,i]=K$ \\\quad\\if both elements in the bracket have the same index.\\
\end{defi}

\begin{defi}
The \textup{octonionic Heisenberg Group} $H^n_\OO$ of dimension $8n+7$ is as manifold
$$H^n_\OO\defeq \OO^n\times \opna{Im}\OO$$
where $\OO$ denotes the Cayley octonions. The group law is given by

$$(z,x)(w,y)\defeq (z+w,x+y-\frac{1}{2}\sum_{i=1}^n\opna{Im}(z_i\overline{w_i}))\ .$$
\quad\\
This is a $2$-step nilpotent Lie group with (real) Lie algebra $\mathfrak h^n_\OO=V_1\oplus V_2$ where $V_1=\OO^n,V_2=\opna{Im}\OO$ and with the bracket %
$$[(Z,X),(W,Y)]=(0,\sum_{i=1}^n\opna{Im}(Z_i\overline{W_i}))\ .$$
\\
It can be seen as the unique simply connected Lie group with Lie algebra generated by 
\begin{align*} B\defeq \{&d_1,\!...,d_n,e_1,\!...,e_n,f_1,\!...,f_n,g_1,\!...,g_n,h_1,\!...,h_n,i_1,\!...,i_n,j_1,\!...,j_n,k_1,\!...,k_n,\\&E,F,G,H,I,J,K \}\end{align*}
and with the only non-trivial brackets of the generators\\\quad\\
$[a,d]=A\\ \text{for } a\in \{e_1,...,e_n,f_1,...,f_n,g_1,...,g_n,h_1,...,h_n,i_1,...,i_n,j_1,...,j_n,k_1,...,k_n\},$ \\
(here $A$ denotes the capital letter of the choice of $a$)
\\\quad\\ and\\\quad\\
 $ [i,f]=[k,h]=[j,g]=E$\\$ [e,i]=[j,h]=[g,k]=F$\\$  [k,f]=[e,j]=[h,i]=G$\\$ [i,g]=[f,j]=[e,k]=H$\\$ [g,h]=[f,e]= [k,j]=I$\\$  [h,f]=[g,e]=[i,k]=J$\\$ [f,g]=[e,h]=[j,i]=K$ \\\quad\\if both elements in the bracket have the same index.
\end{defi}

The above defined generalised Heisenberg Groups are not only abstractly constructed Lie groups. They rather arise in geometry as natural generalisations of the complex Heisenberg Groups. For seeing this one has to consider the complex hyperbolic spaces $\opna{SU}(n,1)/\opna S(\opna U(n)\times \opna U(1))$, the quaternionic hyperbolic spaces $\opna{Sp}(n,1)/(\opna{Sp}(n) \times \opna{Sp}(1))$ and the Cayley plane $\opna F_{4(-20)}/\opna{SO}(9)$ of real dimensions $2n$, $4n$ respectively $16$. In these spaces the horospheres are biLipschitz equivalent to the Heisenberg Groups $H^{n-1}_\CC$, $H^{n-1}_\HH$ respectively $H^1_\OO$. So they are of more general interest, for example for the geometry of non-cocompact lattices in the above mentioned hyperbolic spaces.\\
If one adds the real hyperbolic space $\opna{SO}(n,1)/\opna{SO}(n)$ of dimension $n$  to the above list, the horospheres become biLipschitz equivalent to the \textit{"real Heisenberg Group"} $(\RR^{n-1},+)$. So, by the classification of rank 1 symmetric spaces (see for example \cite{rank1}), the Heisenberg Groups can be seen as horospheres in rank 1 symmetric spaces of non-compact type.

\begin{prop} \label{qoHeis}
Let $G$ be the quaternionic Heisenberg Group $H^n_\HH$ or the octonionic Heisenberg Group $H^n_\OO$ and $\mf g$ the respective Lie algebra. Then there is an $n$-dimensional $\Omega$-regular, $\Omega$-isotropic subspace $S\subset V_1$ of the first layer of the grading \mbox{$\mf g \!=\! V_1 \oplus [\mf g, \mf g]$.} 

\end{prop}

\begin{proof} In the quaternionic case we define $Z_1=I$, $Z_2=J$, $Z_3=K$ and $\eta_i=Z_p^*$, where $Z_p^*$ denotes the dual form to $Z_p$. Then we get by the general formula for left-invariant differential $m$-forms on Lie groups
$$(m+1)!(\opna d \gamma)(Y_0,...,Y_m)=\sum_{i<j}(-1)^{i+j+1}\gamma([Y_i,Y_j],X_0,...,\hat Y_i,...,\hat Y_j,...,Y_m)$$
the following components of the curvature form $\Omega$:
$$\omega_i=\opna d\eta_i=\frac{1}{2}\cdot\sum_{[X_l,X_m]=Z_i} X_l^*\wedge X_m^* \quad \text{ where } X_q\in \{h_q,i_q,j_q,k_q\}.$$
We choose $S=\langle h_1,...,h_{n} \rangle \subset  V_1$, which is $\Omega$-isotropic as $[h_u,h_v]=0$ and therefore $\omega_i(h_u,h_v)=0$ for all $u,v \in \{1,...,n\}$.\\
It remains to give for every choice of $\sigma_{pq}$ a solution $\xi$ for $\omega_p(\xi,h_q)=\sigma_{pq}$ for $p=1,2,3$ and $q=1,...,n$. 
Let $X_{pq}$ denote the unique element in $\{i_q,j_q,k_q\}$ with $[X_{pq},h_q]=Z_p$ . Then one can check by a short computation, that such a solution is given by the following element:
$$\xi= \sum_{p,q} \sigma_{pq}X_{pq} \ .$$
In the octonionic case we define $Z_1=E$, $Z_2=F$, $Z_3=G$, $Z_4=H$, $Z_5=I$, $Z_6=J$, $Z_7=K$ and $\eta_p=Z_p^*$, where $Z_p^*$ denotes the dual form to $Z_p$.  Then we get by the general formula for left-invariant differential $m$-forms on Lie groups the following components of the curvature form $\Omega$: 
$$\omega_i=\opna d\eta_i=\frac{1}{2}\cdot\sum_{[X_l,X_m]=Z_i} X_l^*\wedge X_m^* \quad \text{ where } X_p\in \{e_q,f_q,g_q,h_q,i_q,j_q,k_q\}$$
We choose $S=\langle d_1,...,d_{n} \rangle \subset  V_1$, which is $\Omega$-isotropic as $[d_u,d_v]=0$ and therefore $\omega_i(d_u,d_v)=0$  for all $u,v \in \{1,...,n\}$.\\
It remains to give for every choice of $\sigma_{pq}$ a solution $\xi$ for $\omega_p(\xi,d_q)=\sigma_{pq}$ for $p=1,2,3,4,5,6,7$ and $q=1,...,n$. 
Let $X_{pq}$ denote the unique element in $\{e_q,f_q,g_q,h_q,i_q,j_q,k_q\}$ with  $[X_{pq},d_q]=Z_p$  .
Then one can check by a short computation, that such a solution is given by the following element:
$$\xi= \sum_{p,q} \sigma_{pq}X_{pq} \ .$$
\end{proof}

\begin{cor}\label{Cor1}
Let $H^n_\HH$ be the $4n+3$ dimensional quaternionic Heisenberg Group. Then holds:
\begin{enumerate}[i)]
\item $F^{j+1}_{H^n_\HH}(l) \sim l^\frac{j+1}{j}$\quad for \ $j<n$,
\item $F^{n+1}_{H^n_\HH}(l) \preccurlyeq l^\frac{n+2}{n}$\quad,
\item  $F^{m+1}_{H^n_\HH}(l) \sim l^\frac{m+4}{m+3}$\quad for \ $3n+3 < m < 4n+3$.\\
\end{enumerate}
\end{cor}

\begin{cor}\label{Cor2}
 Let $H^n_\OO$ be the $8n+7$ dimensional octonionic Heisenberg Group. Then holds:
\begin{enumerate}[i)]
\item $F^{j+1}_{H^n_\OO}(l) \sim l^\frac{j+1}{j} $\quad for \ $j<n$,
\item $F^{n+1}_{H^n_\OO}(l) \preccurlyeq l^\frac{n+2}{n}$\quad,
\item $F^{m+1}_{H^n_\OO}(l) \sim l^\frac{m+8}{m+7}$\quad for \ $7n+7 < m < 8n+7$.
\end{enumerate}
\end{cor}
By the above Proposition \ref{qoHeis}, the quaternionic and the octonionic Heisenberg Groups fulfil the conditions of Theorem \ref{Thm3} and Theorem \ref{Thm7} as they are $2$-step nilpotent. So Corollary \ref{Cor1} and Corollary \ref{Cor2} follow by Theorem \ref{Thm3} as the Hausdorff-dimension of $H^n_\HH$ is $4n+3$ and the Hausdorff-dimension of $H^n_\OO$ is $8n+7$.

The proof of Theorem \ref{Thm5} and Theorem \ref{Thm6} (respectively Theorem \ref{Thm7}) on the higher divergence functions only uses the bounds on the filling functions. So these theorems remain true if one replaces the conditions of them by the bounds on the filling functions established in \mbox{Theorem \ref{Thm1}} and Theorem \ref{Thm2} (respectively Theorem \ref{Thm3}). Using the  bounds for the complex Heisenberg Groups (computed in \cite{Young1} and \cite{YoungII}) we get the following behaviour of the higher divergence functions:

\begin{cor}\label{Cor3}
Let $H^n_\CC$ be the $2n+1$ dimensional complex Heisenberg Group. Then holds:
\begin{enumerate}[i)]
\item $ \opna{Div}_{H^n_\CC}^j(r) \succcurlyeq r^{j+1}$\quad for \ $j<n$,
\item $ \opna{Div}_{H^n_\CC}^n(r) \succcurlyeq r^{n+2}$\quad,
\item $\opna{Div}^{m}_{H^n_\CC}(r) \sim r^\frac{(m+2)m}{m+1}$\quad for \ $n+1 \le m < 2n$.
\end{enumerate}
\end{cor}

And for the quaternionic and octonionic Heisenberg Groups we obtain:

\begin{cor}\label{Cor4}
Let $H^n_\HH$ be the $4n+3$ dimensional quaternionic Heisenberg Group. Then holds:
\begin{enumerate}[i)]
\item $ \opna{Div}_{H^n_\HH}^j(r) \succcurlyeq r^{j+1}$\quad for \ $j<n$,

\item $\opna{Div}^{m}_{H^n_\HH}(r) \sim r^\frac{(m+4)m}{m+3}$\quad for \ $3n+3 < m < 4n+2$.\\
\end{enumerate}
\end{cor}
\begin{cor}\label{Cor5}
 Let $H^n_\OO$ be the $8n+7$ dimensional octonionic Heisenberg Group. Then holds:
\begin{enumerate}[i)]
\item $\opna{Div}_{H^n_\OO}^j(r) \succcurlyeq r^{j+1}$\quad for \ $j<n$,

\item $\opna{Div}^{m}_{H^n_\OO}(r) \sim r^\frac{(m+8)m}{m+7}$\quad for \ $7n+7 < m < 8n+6$.
\end{enumerate}
\end{cor}

\mbox{Corollary \ref{Cor4}} and Corollary \ref{Cor5} follow by Theorem \ref{Thm7}.
\quad\\
For the proof of Corollary \ref{Cor3} concerning the higher divergence functions of the complex Heisenberg Group $H^n_\CC$, we use \cite[Corollary 1.1]{Young1} and \mbox{\cite[Theorem 1]{YoungII}}, which together state the following behaviour of the filling functions: 
$$F_{H^n_\CC}^{j+1}(\ell) \sim \ell^\frac{j+1}{j} \quad\text{ for } j<n,$$
$$F_{H^n_\CC}^{n+1}(\ell) \sim \ell^\frac{n+2}{n}  \quad ,\hspace{1.5cm} $$
$$F_{H^n_\CC}^{j+1}(\ell) \sim \ell^\frac{j+2}{j+1}  \quad\text{ for } j>n. $$
With Proposition \ref{Divlow} and Proposition \ref{Divup} this proves Corollary \ref{Cor3}.

In the above Corollaries \ref{Cor1} and \ref{Cor2} concerning the filling functions of the quaternionic and the octonionic Heisenberg Groups we only stated upper bounds in dimension $n$. This is due to the fact, that the technique for the lower bounds, used by Burillo in \cite{Burillo} to compute the lower bounds in dimension $n$ for the complex Heisenberg Groups, is very special. Indeed it can't be generalised to stratified nilpotent Lie groups, not even to the octonionic Heisenberg Groups as the example of $H^1_\OO$ (discussed below) shows. 

By the discussion following Theorem \ref{Thm1}, one can see, that $n$ is the maximal possible dimension of an $\Omega$-isotropic, $\Omega$-regular subspace $S$ of $V_1$ in the grading of the Lie algebra of the quaternionic respectively octonionic Heisenberg Group $H^n_\HH$ respectively $H^n_\OO$. This follows from the equalities
$$ \dim V_1 - n =4n-n=3n = n(4n+3 -4n)=n(\dim \mf g - \dim V_1)$$ for the quaternionic case,
and
$$ \dim V_1 - n =8n-n=7n = n(8n+7 -8n)=n(\dim \mf g - \dim V_1)$$ for the octonionic case.\\
And as the left hand side is strictly decreasing and the right hand side is strictly increasing in the dimension of the $\Omega$-regular, $\Omega$-isotropic horizontal subspace, the necessary condition for the existence of such a subspace  is not satisfied for any dimension greater than $n$.

\begin{lem}\label{oHeis}
There is no $(n+1)$-dimensional $\Omega$-isotropic subspace $S\subset V_1$ in the  octonionic Heisenberg Group  $H^{n}_\OO$.
\end{lem}
\begin{proof}
We prove this in two steps:\vspace*{-3mm}
\begin{enumerate}[]
\item[\underline{Step 1):}] Every $n$-dimensional $\Omega$-isotropic subspace $W \subset V_1$ is $\Omega$-regular.\\\quad\\
Let $W=\langle w_1,...,w_n\rangle_\mathbb R$ be an $n$-dimensional $\Omega$-isotropic subspace spanned by
$$w_p=(w_p^{d_1},w_p^{d_2},....,w_p^{k_n}) \quad \text{for }1\le p \le n,$$
with coordinates with respect to the basis $\{d_1,...,k_n\}$ of the first layer $V_1$ of the grading of Lie algebra. \\
Let $Z_1=E,\ Z_2=F$, $Z_3=G$, $Z_4=H$, $Z_5=I$, $Z_6=J$ and $Z_7=K$.
Then holds for $l \in \{1,2,3,4,5,6,7\}$:
%
\begin{align*}
\opna d\eta_l(\_,w_p )=& \sum_{q \atop [X_m,d_q]=Z_l} w_p^{d_q} X_m^*+\sum_{q \atop [X_m,e_q]=Z_l} w_p^{e_q} X_m^*+\sum_{q \atop [X_m,f_q]=Z_l} w_p^{f_q} X_m^*\\&+\sum_{q \atop [X_m,g_q]=Z_l} w_p^{g_q} X_m^*+\sum_{q \atop [X_m,h_q]=Z_l} w_p^{h_q} X_m^*+\sum_{q \atop [X_m,i_q,]=Z_l} w_p^{i_q} X_m^*\\&+\sum_{q \atop [X_m,j_q]=Z_l} w_p^{j_q} X_m^*+\sum_{q \atop [X_m,k_q]=Z_l} w_p^{k_q} X_m^*
\end{align*}
So we get:

\end{enumerate}
$$\opna d \eta_1( \_ ,w_p)= \sum_q \big (w_p^{d_q}e_q^* + w_p^{f_q}i_q^*+w_p^{h_q}k_q^*+w_p^{g_q}j_q^* -w_p^{e_q}d_q^* - w_p^{i_q}f_q^*-w_p^{k_q}h_q^*-w_p^{j_q}g_q^*\big)$$
$$\opna d \eta_2(\_ ,w_p )= \sum_q \big (w_p^{d_q}f_q^* + w_p^{i_q}e_q^*+w_p^{h_q}j_q^*+w_p^{k_q}g_q^* -w_p^{f_q}d_q^* - w_p^{e_q}i_q^*-w_p^{j_q}h_q^*-w_p^{g_q}k_q^*\big)$$
$$\opna d \eta_3(\_ ,w_p )= \sum_q \big (w_p^{d_q}g_q^* + w_p^{f_q}k_q^*+w_p^{j_q}e_q^*+w_p^{i_q}h_q^* -w_p^{g_q}d_q^* - w_p^{k_q}f_q^*-w_p^{e_q}j_q^*-w_p^{h_q}i_q^* \big)$$
$$\opna d \eta_4(\_ ,w_p )= \sum_q \big (w_p^{d_q}h_q^* + w_p^{g_q}i_q^*+w_p^{j_q}f_q^*+w_p^{k_q}e_q^* -w_p^{h_q}d_q^* - w_p^{i_q}g_q^*-w_p^{f_q}j_q^*-w_p^{e_q}k_q^* \big)$$
$$\opna d \eta_5(\_ ,w_p)= \sum_q \big (w_p^{d_q}i_q^* + w_p^{h_q}g_q^*+w_p^{e_q}f_q^*+w_p^{j_q}k_q^* -w_p^{i_q}d_q^* - w_p^{g_q}h_q^*-w_p^{f_q}e_q^*-w_p^{k_q}j_q^* \big)$$
$$\opna d \eta_6(\_ ,w_p )= \sum_q \big (w_p^{d_q}j_q^* + w_p^{f_q}h_q^*+w_p^{e_q}g_q^*+w_p^{k_q}i_q^* -w_p^{j_q}d_q^* - w_p^{h_q}f_q^*-w_p^{g_q}e_q^*-w_p^{i_q}k_q^* \big)$$
$$\opna d \eta_7(\_ ,w_p )= \sum_q \big (w_p^{d_q}k_q^* + w_p^{g_q}f_q^*+w_p^{h_q}e_q^*+w_p^{i_q}j_q^* -w_p^{k_q}d_q^* - w_p^{f_q}g_q^*-w_p^{e_q}h_q^*-w_p^{j_q}i_q^* \big)$$
\begin{enumerate}[]
\vspace{-7mm}
\item
As $W$ is $\Omega$-isotropic, the curvature form 
$$\Big(\Omega(w_p, \_)\Big)_p=\Big(\big(\opna d \eta_1(w_p, \_ ),\opna d \eta_2(w_p, \_ ),...,\opna d \eta_7(w_p, \_ )\big)\Big)_p$$
can be interpreted as a linear isomorphism $V_1/W \cong \mathbb R^{7n} \to \mathbb R^{7n}$.\\
For $\sigma=(\sigma_{pq}) \in \mathbb R^{7n}$ this leads, as the $w_p$ are linear independent, to a system $A\xi=\sigma$ of linear equations with an invertible matrix $A\in \mathbb R^{7n\times 7n}$. So there always exists a solution and $W$ is $\Omega$-regular.
\item[\underline{Step 2):}] There is no $(n+1)$-dimensional $\Omega$-isotropic subspace $S\subset V_1$.

Assume there is an $(n+1)$-dimensional $\Omega$-isotropic subspace 
$S=\langle s_1,...,s_{n+1} \rangle_\mathbb R$ of  $V_1$.
Then $W=\langle s_1,...,s_n \rangle_\mathbb R$ is $\Omega$-isotropic and $n$-dimensional. Claim 1 implies that $W$ is $\Omega$-regular. So the linear map 
$$\Omega^W_\bullet:V_1 \to \opna{Hom}(W,\mf g/V_1),\ X \mapsto \Omega(X, \_)$$
is surjective and vanishes on $S$.
Therefore: 
\begin{align*} 
7n-1 &= 8n-(n+1) = \dim(V_1/S) \\
&\ge \dim(V_1/\opna{kern}(\Omega_\bullet^W)=\dim (\opna{image}(\Omega_\bullet^W))\\
&=\dim (\opna{Hom}(W,\mf g/V_1))=n(8n+7-8n)=7n
\end{align*}
But this is a contradiction and such an $S$ can't exist.
\end{enumerate}
\vspace{-11mm}
\end{proof}

By Proposition \ref{qoHeis} and Lemma \ref{oHeis} we have condition $b)$ of Theorem \ref{Thm4} satisfied for the octonionic Heisenberg Group. So we get:

\begin{cor}\label{Cor6}\label{superO}
Let $H^n_\OO$ be the octonionic Heisenberg Group of dimension $8n+7$.
Then holds: 
$$F^{n+1}_{H^n_\OO}(l) \succnsim l^\frac{n+1}{n} \ .$$
\end{cor}

%
%

\subsection{Lattices in rank 1 symmetric spaces}\label{app}
\vspace*{-5mm}

For $n\ge 2$  the Corollaries \ref{Cor1} and \ref{Cor2} show that the $2$-dimensional filling function $F^2$ is of quadratic type for $H^n_\HH$ and $H^n_\OO$. . This means that a sub-quadratic Dehn function has to be linear, as one can use the fact that there is a gap between linear and quadratic Dehn functions (see \cite{Bowditch}). By the relation $\delta^1 \preccurlyeq F^2$ and the fact that the above groups are not hyperbolic (and therefore can't have linear Dehn functions), their Dehn functions $\delta^1$ are quadratic, too. 
Pittet computed the Dehn function of $H^1_\HH$ (which is cubic) in \textup{\cite{Pittet}} and so the Dehn function of $H^n_\HH$ is known for all $n\in \mathbb N$. Unfortunately there is an error in Pittet's computation of the Dehn function of $H^1_\OO$ (mentioned in \textup{\cite{LeuzingerPittet}}). And annoyingly this error can't be repaired, because the used proposition in \textup{\cite{Pittet}} needs a $2$-form $\omega$ of the shape
$$\omega=\sum_{x_i\in B_1,Y_i\in B_2}\alpha_i(Y_i^*\wedge x_i^*) \ne 0$$
with differential 
$$\opna{d}\omega=\sum_{x_i\in B_1,Y_i\in B_2}\alpha_i \ \opna d(Y_i^*\wedge x_i^*)=0$$
where $B_1=\{d,e,f,g,h,i,j,k\}$ and $B_2=\{E,F,G,H,I,J,K\}$. The condition on the differential leads to a system of linear equations with no non-trivial solution. So there is no such $2$-form and Pittet's technique doesn't work for $H^1_\OO$.\\ One could think, that Burillo's technique \cite{Burillo} could solve this problem, but this leads to the same requirement of a $2$-form with the properties described above. 
This means the $2$-dimensional filling function of $H^1_\OO$ is still unknown, while the $2$-dimensional  filling function of $H^n_\OO$ is known for all $n\ge 2$. 
Nevertheless our results have an application to the (higher dimensional) Dehn functions of non-uniform lattices is the complex and quaternionic hyperbolic spaces:

\begin{cor}\label{hyp}
Let $n \in \mathbb N_{\ge 3}$ and $X$ be the complex hyperbolic space $\opna{SU}(n,1)/\opna S(\opna U(n)\times \opna U(1))$ of dimension $2n$ or the quaternionic hyperbolic space $\opna{Sp}(n,1)/(\opna{Sp}(n) \times \opna{Sp}(1))$ of dimension $4n$. Further let $\Gamma$ be a group acting properly discontinuously by isometries, such that the quotient space $X / \Gamma$ is of finite volume, but not compact. Then holds: 
$$\delta^j_\Gamma(l) \sim l^{\frac{j+1}{j}}  \qquad\text{ for } 1\le j< n-1.$$
\end{cor}
\vspace*{-6mm}
\begin{proof}
For the complex hyperbolic case this is proved in \cite[Theorem 5]{Leuzinger14}.\\ 
For the quaternionic case we use, as in the proof of \cite[Theorem 5]{Leuzinger14}, that $\Gamma$ acts geometrically on a space $X_0$ obtained from the quaternionic hyperbolic space $\opna{Sp}(n,1)/\opna{Sp}(n) \times \opna{Sp}(1)$ by removing a $\Gamma$-invariant family of  horoballs. By the Lemma of \v{S}varc-Milnor, $\Gamma$ is quasi-isometric to $X_0$ and so 
$\delta^j_\Gamma \sim \delta^j_{X_0} \sim F^{j+1}_{X_0}$ for all $1 \le j \le n-1$.
The filling functions of $X_0$ are equivalent to the filling functions of any of its boundary components (see \cite[Theorem 5]{Leuzinger14},\cite[5.D.(5)(c)]{GGT}), which are horospheres in $\opna{Sp}(n,1)/\opna{Sp}(n) \times \opna{Sp}(1)$.
The horospheres in the quaternionic hyperbolic space $\opna{Sp}(n,1)/\opna{Sp}(n) \times \opna{Sp}(1)$ are biLipschitz equivalent to the quaternionic Heisenberg Group $H^{n-1}_\HH$.
By Corollary \ref{Cor1} we get $\delta^j_\Gamma(l) \sim F^{j+1}_{X_0}(l) \sim l^\frac{j+1}{j}$ for all $1\le j \le n-1$.
\end{proof}

\begin{rem}
Let $X$ be the complex hyperbolic plane $\opna{SU}(2,1)/\opna S(\opna U(2)\times \opna U(1))$ of dimension $4$ or the quaternionic hyperbolic  plane $\opna{Sp}(2,1)/(\opna{Sp}(2) \times \opna{Sp}(1))$ of dimension $8$ and $\Gamma$  a group acting properly discontinuously on $X$ by isometries, such that the quotient space has finite volume but is not compact. It is proved in \textup{\cite{Pittet}} that:
$$\delta_\Gamma(l) \sim\delta^1_\Gamma(l) \sim l^3 \ .$$
\end{rem}


\vspace*{-5mm}
\section{Horizontal approximation}\label{S4}
\vspace*{-5mm}

To apply Theorem \ref{ThmYoung} and Theorem \ref{ThmYoungII} of Young, we needed to approximate a triangulation $(\tau,f)$ of a stratified nilpotent Lie group $G$ by a $(k+1)$-horizontal map $\psi: \tau \to G$. In this chapter we prove the techniques that enabled us to do this.\\
We consider $G$ equipped with the sub-Riemannian metric $\opna d_c$. We further will use a lot the properties $\Omega$-isotropic and $\Omega$-regular, so remember the definition of the curvature form $\Omega$ (see Section \ref{SectionSNLG}).

\vspace*{-5mm}
\subsection{Some definitions}
\vspace*{-5mm}

Let $W$ be a simplicial complex and let $V$ be a smooth manifold. We call a map $f:W \to V$ smooth, immersion or horizontal if the respective property holds for $f$ restricted to each single simplex of $W$.

\begin{defi}
Let $W'$ be a simplicial complex and let $V$ be a smooth manifold. A  map $f:W' \to V$ is a \textup{folded immersion} if $W'$ has a locally finite covering by compact subcomplexes, $W'=\bigcup W'_i$, such that $f$ is smooth on each simplex of $W'$ and sends each $W'_i$ homeomorphically to a smooth compact submanifold (with boundary) of $V$. \\
If $V=G$ is a stratified nilpotent Lie group, we call a folded immersion $f$  \textup{horizontal} (and/or \textup{$\Omega$-regular}) if $f(W'_i)$ is horizontal (and/or $\Omega$-regular) for all $i$.
\end{defi}

\begin{defi}
Let $G$ be a stratified nilpotent Lie group and let $T$ and $T'$ be  $m$-dimensional simplicial complexes. We say $f':T' \to G$ \textup{approximates} $f:T\to G$, if for every neighbourhood $U\subset G\times T$ of the graph $\{(f(x),x) \mid x\in T\}$ of $f$ and every neighbourhood $V \subset T\times T$ of the diagonal $\{(x,x)\mid x\in T\}$ there are proper homotopy equivalences $\varphi: T \to T'$ and $\varphi': T' \to T$ such that:\vspace{-3mm}
\begin{enumerate}[i)]
\item The graph of $\varphi \circ \varphi'$ is contained in $V$.\vspace{-3mm}
\item The graph of $f' \circ \varphi$ is contained in $U$.\vspace{-3mm}
\end{enumerate}
\end{defi}

%
%

\subsection{The folded approximation theorem}
\vspace*{-5mm}
To construct approximations of maps into stratified nilpotent Lie groups by horizontal ones, we need the following Lemma:

\begin{lem}[local h-principle, {\cite[4.2.A$'$]{Gromov}, compare \cite{D'Ambra} and \cite[2.3.2]{PDR}}] \label{localhlem}
Let $G$ be a stratified nilpotent Lie group with grading $\mf g=V_1 \oplus ... \oplus V_d$ of the Lie algebra. Further let $T$ be an $m$-dimensional simplicial complex. 
Then the sheaf of horizontal, $\Omega$-regular, smooth immersions $f:T \to G$ is micro-flexible and satisfies the local h-principle. In particular, for every $\Omega$-regular, $\Omega$-isotropic $m$-dimensional subspace $S\subset V_1$ and every $ g\in G$ there exists a germ of smooth integral submanifolds $W_i \subset G$ at $g$ with $T_g(W_i)=dL_gS$.
\end{lem}

This is a corollary of the Main Theorem of \textup{\cite[2.3.2]{PDR}} as the differential operator which sends smooth maps $f:T \to G $ to the induced forms $\{f^*(\eta_i)\}_i$ is infinitesimal invertible on $\Omega$-regular horizontal immersions (compare \textup{\cite[2.3.1]{PDR}}).

\begin{prop}[Folded Approximation Theorem, {\cite[4.4]{Gromov}}] \label{approx} 
Let $G$ be a stratified nilpotent Lie group and let $T$ be a $m$-dimensional simplicial complex. 
Then a continuous map $f_0: T \to G$ admits an approximation by folded horizontal $\Omega$-regular immersions $f':T'  \to G$ if and only if there is a continuous map $T \ni x \mapsto S_x $, where $S_x$ is the translate of an $\Omega$-regular, $\Omega$-isotropic $m$-dimensional subspace $S\subset V_1$.
\end{prop}

\begin{proof}
Let $f_0:T \to G$ be continuous and let $f':T' \to G$ be a folded horizontal $\Omega$-regular immersion approximating $f_0$. Then, by definition, there is  a homotopy equivalence $\varphi: T \to T'$. The pullback under $\varphi$ of the folded tangent bundle over $T'$ provides a continuous map $T \ni x \mapsto S_x \subset dL_{f_0(x)}V_1$. This proves the "only if" part. \\
Let's turn towards the "if" direction:
As the sheaf of horizontal $\Omega$-regular immersions $f:T\to G$ are microflexible, $\opna{Diff}(V)$-invariant and by Lemma \ref{localhlem} locally integrable, we can use  \cite[Theorem 13.4.1]{hprinciple} to prove the above Proposition for the $(m-1)$-skeleton $T^{(m-1)}$ of $T$:
$T^{(m-1)}$ is a subpolyhedron of codimension 1 and 
$$F_0: \mc Op(T^{(m-1)})\to \mc R, \ x \mapsto (f_0(x),S_x)$$
is a formal solution near $T^{(m-1)}$. With \cite[Theorem 13.5.1]{hprinciple} there is a genuine solution $F:\mc Op(T^{(m-1)})\to \mc R$ homotopic to $F_0$ and such that $f=\opna{bs}F$ is arbitrary close to $f_0$.\\
It remains to prove the existence of the approximation in the top-dimensional case under the assumption that $f_0$ is an horizontal $\Omega$-regular immersion near the $(m-1)$-skeleton of $T$.\\
We can treat each simplex $\Delta^m\subset T$ separately. So let $f_0: \Delta^m \to G$ be an horizontal $\Omega$-regular immersion near $\partial \Delta^m \subset T^{(m-1)}$.\\
We consider $\Delta^m$ as $\Delta^m= (\partial \Delta^m \times [0,1]) / \sim$ , with $\sim$ the equivalence relation defined by $(x,1)\sim (y,1) \ \forall x,y \in \partial \Delta^m$. Denote by $\Delta(t)$ the layer $\partial \Delta^m \times \{t\}$.
\begin{figure}[h]
\centering
\includegraphics[width=50mm]{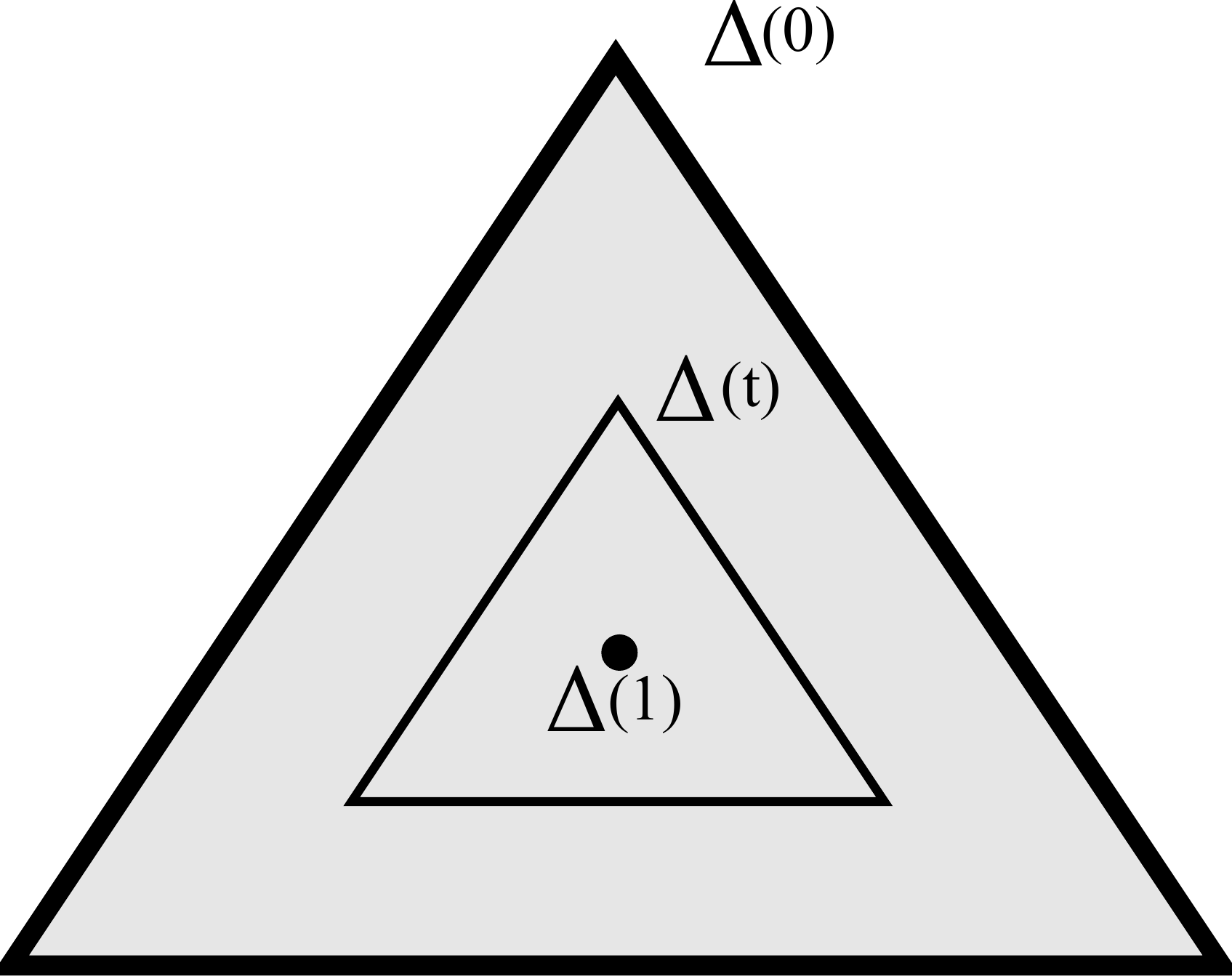}
\caption{ The layers of $\Delta^m$ for $m=2$.}
\end{figure}
\quad\\
Each of this layers $\Delta(t)$ has dimension $\le m-1$. We look at the maps 
$$F_0^t: \mc Op(\Delta(t)) \to \mc R, x \to (f_0(x), S_x)$$
and by \cite[Theorem 13.5.1]{hprinciple} we get horizontal $\Omega$-regular immersions $f^t:\mc Op(\Delta(t)) \to G$ close to $f_0$ (for $t=0$ we take $f^0=f_0$). With \cite[Theorem 13.5.1]{hprinciple}, or more exactly with the parametric version of it, we can choose the family $\{f^t\}$ to be continuous in $t$. Now let $\varepsilon >0$ be sufficiently small, such that the $2\varepsilon$-neighbourhood of the $t$-layer 
$$N_{2\varepsilon}(t)\defeq (\partial \Delta^m\times ((t-2\varepsilon,t+2\varepsilon)\cap[0,1]))/\sim$$
is contained in $\mc Op(\Delta(t))$ for all $t$.
We define the (holonomic) homotopy
$$H^0: \big(N_\frac{\varepsilon}{3}(0) \cup N_\frac{\varepsilon}{3}(\varepsilon)\big)\times [0,1] \to G\ ,$$
$$(x,t) \mapsto \begin{cases} f^0(x) , \text{ if } x \in  N_\frac{\varepsilon}{3}(0) \\                                            f^t(x) , \text{ if } x \in  N_\frac{\varepsilon}{3}(\varepsilon) \end{cases}\ .$$
The microflexibility gives a positive $t_1\in (0,1]$ such that we can extend $H^0$ on $\mc Op(\Delta(0))\times [0,t_1]$. Then we replace $f^0$ by $H_{t_1}^0$.
We define the (holonomic) homotopy
$$H^1: \big(N_\frac{\varepsilon}{3}(t_1-\varepsilon) \cup N_\frac{\varepsilon}{3}(t_1+\varepsilon)\big)\times [0,1] \to G\ , $$
$$ (x,t) \mapsto \begin{cases} f^{t_1}(x) , \text{ if } x \in  N_\frac{\varepsilon}{3}(t_1+\varepsilon) \\                                            f^{t_1+t}(x) , \text{ if } x \in  N_\frac{\varepsilon}{3}(t_1-\varepsilon) \end{cases}$$
where $f^s=f^1$ for all $s\ge 1$.\\
The microflexibility gives a positive $t_2\in (0,1]$ such that we can extend $H^1$ on $\mc Op(\Delta(t_1))\times [0,t_2]$. We replace $f^{t_1}$ by $H_{t_2}^1$.
We define the (holonomic) homotopy
$$H^2: \big(N_\frac{\varepsilon}{3}(t_1+t_2-\varepsilon) \cup N_\frac{\varepsilon}{3}(t_1+t_2+\varepsilon)\big)\times [0,1] \to G\ , $$
$$ (x,t) \mapsto \begin{cases} f^{t_1+t_2}(x) , \text{ if } x \in  N_\frac{\varepsilon}{3}(t_1+t_2-\varepsilon) \\                                            f^{t_1+t_2+t}(x) , \text{ if } x \in  N_\frac{\varepsilon}{3}(t_1+t_2+\varepsilon) \end{cases}\ .$$
The microflexibility gives a positive $t_3\in (0,1]$ such that we can extend $H^2$ on $\mc Op(\Delta(t_1+t_2))\times [0,t_3]$. We replace $f^{t_1+t_2}$ by $H_{t_3}^2$.\\
We continue this procedure until we reach $f^1$. As the number $t_1$ depends continuously on the layers, i.e. on $t\in [0,1]$, there is the minimum $ \min \{t_i\} >0$ and so we only need finitely many steps (less than $\lceil\frac{1}{\min\{t_i\}}\rceil$ many).
We define $T'\defeq \bigcup \overline{N_{\frac{2}{3}\varepsilon}(t_i)}$ as the disjoint union of the closed $\frac{2}{3}\varepsilon$-neighbourhoods of the $\Delta(t_i)$. Further we make identifications corresponding to the intersections of their images under the $H^i_{t_{i+1}}$ and give it a simplicial structure such that the $\Delta(t_i)$ are contained in the $(m-1)$-skeleton. $T'$ is obviously homotopy equivalent to $T$.  Then we get the approximating map as 
$$f'\defeq \bigcup H^i_{t_{i+1}} :T' \to G$$ 
where $f'_{|\overline{N_\varepsilon(t_i)}}= H^i_{t_{i+1}}$.
\end{proof}

\begin{figure}[h]
\centering
\includegraphics[width=120mm]{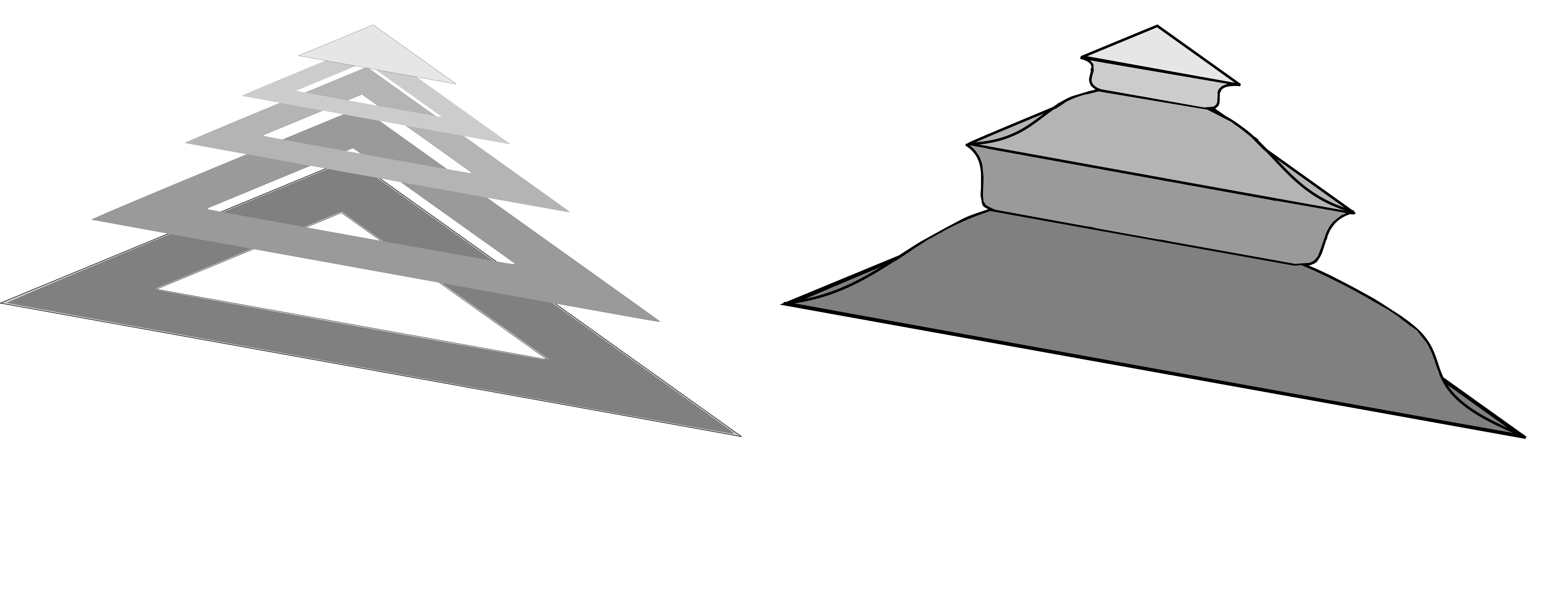}
\caption{Schematical illustration of he images of the $f^{t_i}$ (left) and the images of the $H^i_{t_{i+1}}$ (right) for $m=2$.}
\end{figure}

\quad\\
The above proof shows, that we need \textit{folded} immersions only in the top-dimension (i.e. in dimension $m$, the dimension of the $\Omega$-regular, $\Omega$-isotropic  subspace $S$). In  lower dimensions we can consider $T$ as a subcomplex of an $m$-dimensional simplicial complex $\widetilde T$ and the local $h$-principle yields an $\Omega$-regular horizontal immersion 
$$\widetilde f:\mc Op(T) \to G$$
approximating $f_0$ on $T$ and the desired approximation is given by $f\defeq \widetilde f_{|T}$.

\begin{cor}[{\cite[4.4 Corollary]{Gromov}}]\label{LemGr2}
Let $G$ be a stratified nilpotent Lie group with Lie algebra $\mf g$ and $m \in \mathbb N$. Further let $S$ be a $m$-dimensional $\Omega$-isotropic, $\Omega$-regular horizontal subspace of $\mf g$. Then every continuous map $f_0: T \to G$ from an $m$-dimensional simplicial complex $T$ into the stratified nilpotent Lie group $G$ can be approximated by continuous, piecewise smooth, piecewise horizontal maps $f:T \to G$.
\end{cor}
\begin{proof}
We define $S_x=dL_{f_0(x)}S$ and get a horizontal approximation $f':T' \to G$ by Proposition \ref{approx}. As $f'$ approximates $f_0$, there is a homotopy equivalence $\varphi:T \to T'$ and we can use $f=f'\circ \varphi$ as the desired approximation. 
\end{proof}
%
%

\begin{rem}\quad \vspace{-3mm}
\begin{enumerate}[a)]
\item We used horizontal approximations in the Riemannian manifold $(G, \opna d_g)$. For the change from the Carnot-Carath\'eodory metric to the Riemannian metric one can use the fact, that both metrics $\opna d_c$ and $\opna d_g$ induce the same topology on $G$ (see \textup{\cite[Proposition 2.26]{LeDonneRigot}}). What one really needs to transport the above results to $G$ equipped with a left-invariant Riemannian metric is, that every continuous  map $f:T \to (G,\opna d_c)$ is also continuous as map $f:T \to (G, \opna d_g)$. This holds in any case if the by $\opna d_c$ induced topology is finer than the topology of the Riemannian manifold.
This is true, as the identity map $\iota: (G, \opna d_c) \to (G, \opna d_g)$ is $1$-Lipschitz (compare Lemma \ref{1lip}) and therefore continuous. So every open set in $(G,\opna d_g)$ is also open as subset of $(G, \opna d_c)$.
Further is the notion of being horizontal in both cases the same. So the above lemma yields a piecewise horizontal approximation of $f_0$ with respect to the Riemannian metric.
\vspace{-3mm}
\item The above Proposition \ref{approx} and its Corollary \ref{LemGr2} hold true for each microflexible differential relation $\mc R_S$ of $S$-directed immersions on a smooth manifold $M$. The points one has to change are to demand the map $x \mapsto S_x$ to be a continuous map into $S$ and the resulting immersion will be $S$-directed instead of horizontal. Then the proof goes exactly the same way as above, one has just to replace the $G$ by $M$, the translates of the $\Omega$-regular, $\Omega$-isotropic subspaces by $S$ and horizontal by $S$-directed.
\end{enumerate}
\end{rem}


\section{Some open questions}\label{S8}
\vspace*{-5mm}

We computed the filling functions for stratified nilpotent Lie groups under the assumption of the existence of $\Omega$-regular, $\Omega$-isotropic subspaces of the first layer of the Lie algebra. 
Our results suggest a division of the behaviour of the filling functions in a part of Euclidean growth in the low dimensions, strictly sub-Euclidean growth in the highest dimensions and at least one dimension of strictly super-Euclidean growth in between. Whether this is still true without the algebraic condition, is an open question. \\
Burillo proved a cubic lower bound for the filling area function of the group $N_n$ of unipotent upper triangular $n\times n$-matrices for $n \ge 4$. So Gromov's conjecture (see \cite[5.D.]{GGT}) about the first super-Euclidean filling function of a nilpotent Lie group is wrong. The group $N_n$ is $(n-1)$-step nilpotent and Gromov's heuristic argument is mainly based on observations for the complex Heisenberg Groups which are $2$-step nilpotent. So we ask: 

\quest{1}
Does every stratified nilpotent Lie group of nilpotency degree $2$ have Euclidean filling functions up to the maximal dimension of horizontal submanifolds and a super-Euclidean filling function in the dimension above?

An important reason why we are interested in (non-abelian) nilpotent Lie groups is the fact, that they have sectional curvature of both signs, negative and positive, at each point. All spaces of non-positive curvature have no super-Euclidean filling functions, so it would be interesting, if the positive curvature at every point could be seen by the filling functions: 

\quest{2}
Does every (non-abelian) nilpotent Lie group have a super-Euclidean filling function in some dimension?

This question may be easier to answer, if one restricts it to stratified nilpotent Lie groups:

\quest{2b}
Does every stratified nilpotent Lie group have a super-Euclidean filling function in some dimension?

Another question arises for the higher divergence functions of stratified nilpotent Lie groups. We computed the exact growth rate for higher divergence functions in the high dimensions, but in the lower dimensions we only have been able to establish Euclidean lower bounds.

\quest{4}
Are there Euclidean upper bounds for the higher divergence functions of stratified nilpotent Lie groups in low dimensions?
\vspace*{-5mm}

\bibliography{bib}
\bibliographystyle{plain}

\textsc{Karlsruhe Institute of Technology, Karlsruhe, Germany}\\
\hspace*{4mm}\textit{E-mail address:} moritz.gruber@kit.edu

\end{document}